\documentclass[a4paper,12pt, reqno]{amsart}
\usepackage{amsmath}
\usepackage{amssymb}
\usepackage{amsxtra}
\usepackage{amsthm}
\usepackage{mathrsfs}
\usepackage{mathtools}
\usepackage{graphicx}
\usepackage{epstopdf}
\usepackage{subfigure}
\usepackage{float}
\usepackage{xcolor,graphicx,float}
\def\EquationsBySection{\def\theequation
	{\thesection.\arabic{equation}}%
	\@addtoreset{equation}{section}}

\usepackage{indentfirst}
\newtheorem{remark}{Remark}[section]

\newtheorem{theorem}[remark]{Theorem}
\newtheorem{corollary}[remark]{Corollary}
\newtheorem{lemma}[remark]{Lemma}

\newcommand\old[1]{}
\makeatother
\EquationsBySection
\usepackage{xcolor}
\usepackage[colorlinks,citecolor=blue,pagebackref,hypertexnames=false]{hyperref}

\allowdisplaybreaks

\title[Liouville type theorems for semilinear subelliptic systems] %Use the shortened version of the full title
      {Liouville type theorems for  subelliptic systems on the
Heisenberg group with general nonlinearity}
      \author{Rong Zhang, Vishvesh Kumar and Michael Ruzhansky}
\address[ Rong Zhang]{School of Mathematical Sciences,
Nanjing Normal University, 
   Nanjing, 210023, China \newline and \newline Department of Mathematics: Analysis, Logic and Discrete Mathematics, Ghent University, Ghent, Belgium}
\email{zhangrong@nnu.edu.cn}
\address[Vishvesh Kumar]{Department of Mathematics: Analysis, Logic and Discrete Mathematics, Ghent University, Ghent, Belgium}
\email{vishveshmishra@gmail.com / vishvesh.kumar@ugent.be}

\address[Michael Ruzhansky]{Department of Mathematics: Analysis, Logic and Discrete Mathematics, Ghent University, Ghent, Belgium\newline and \newline
School of Mathematical Sciences, Queen Mary University of London, United Kingdom}
\email{michael.ruzhansky@ugent.be}



\allowdisplaybreaks
\begin{document}
\maketitle
\begin{abstract} In this paper, we establish Liouville type results for semilinear subelliptic systems associated with the sub-Laplacian  on the Heisenberg group $\mathbb{H}^{n}$
involving two different kinds of general nonlinearities.
The main technique of the proof is the method of moving planes combined with some integral inequalities replacing the role of maximum principles. As a special case,  we obtain the Liouville theorem for the Lane-Emden system on the Heisenberg group $\mathbb{H}^{n}$, which also appears to be a new result in the literature. 
\end{abstract}

% Enter the first author's name and address:
%\centerline{\scshape Rong Zhang$^{a,b}$,\ Vishvesh Kumar$^{b}$,\ Michael Ruzhansky$^{b,c}$}
%\medskip
%{\footnotesize
% please put the address of the first author
% \centerline{$^{a}$ School of Mathematical Sciences,
%Nanjing Normal University}
%   \centerline{ Nanjing, 210023, China}
%   \centerline{$^b$ Department of Mathematics: Analysis, Logic and Discrete Mathematics, }
 %  \centerline{ Ghent University, B 9000 Ghent, Belgium}
%   \centerline{$^c$ School of Mathematical Sciences, Queen Mary University of London, }
 %  \centerline{ Mile End Rd, Stepney, London E1 4NS, United Kingdom}
%} % Do not forget to end the {\footnotesize by the sign }

%\bigskip

% The name of the associate editor will be entered by an editorial staff
% "Communicated by the associate editor name" is not needed for special issue.
 %\centerline{}

%\tableofcontents 

\section{Introduction}

The result related to the nonexistence of solutions, commonly known as Liouville type theorem, serves as a powerful tool to obtain a priori estimates for a system of equations on bounded domains in the Euclidean space $\mathbb{R}^n$ or in the Heisenberg group $\mathbb{H}^n$. Indeed, in the Euclidean setting, one uses the ``blow up" method, also called the rescaling method, developed by Gidas and Spruck \cite{GS81}. After the blowing up, an equation in a bounded domain becomes an equation in the whole Euclidean space. Now, a Liouville type theorem on $\mathbb{R}^n$ followed by a contradiction argument immediately provides a priori bounds. Following this technique, Birindelli et al. \cite{BDC1} have developed the blow up method for the Heisenberg group and proved a priori bounds for the supremum of the solution to the nonlinear Dirichlet subelliptic problem on a bounded domain in $\mathbb{H}^n$ satisfying an intrinsic cone condition. The idea was to  reduce the problem of
establishing the desired a priori estimates on bounded domains to a problem of Liouville type on the full Heisenberg group. The Liouville theorem for sub-Laplacian on the Heisenberg group is well known, see Folland \cite{foll1,foll2,foll3}, and Geller \cite{Gel83} and also \cite[Section 3.2.8]{FR16}.

The main aim of this paper is to study the nonexistence of positive solutions to the following semilinear subelliptic system on the Heisenberg group
$\mathbb{H}^{n}$:
\begin{equation}\label{a1}
\begin{cases}
\ -\Delta_{\mathbb{H}^{n}}u(\xi)=f(v(\xi)) ,& \xi\in\mathbb{H}^{n},\\
\ -\Delta_{\mathbb{H}^{n}}v(\xi)=g(u(\xi)) ,& \xi\in\mathbb{H}^{n},
\end{cases}
\end{equation}
and 
\begin{equation}\label{z1}
\begin{cases}
\ -\Delta_{\mathbb{H}^{n}}u(\xi)=f(u(\xi),v(\xi)) ,& \xi\in\mathbb{H}^{n},\\
\ -\Delta_{\mathbb{H}^{n}}v(\xi)=g(u(\xi),v(\xi)) ,& \xi\in\mathbb{H}^{n},
\end{cases}
\end{equation}
where $\Delta_{\mathbb{H}^{n}}$ is the sub-Laplacian on $\mathbb{H}^{n}, n \geq 2$ and $f,g$ are two continuous functions. 
It is worth noting that a particular instance of \eqref{a1} is the famous Lane-Emden system. In the Euclidean framework, the  Lane-Emden system was investigated by several prominent mathematicians. We will briefly discuss this development as follows.
%Specifically, for a function $u:\mathbb{R}^{n}\rightarrow \mathbb{R}$, we consider the extension $U:\mathbb{R}^{n}\times[0,\infty)\rightarrow \mathbb{R}$ that satisfies
%$$
%\begin{cases}
%\ div(y^{1-2s}\nabla U)=0,& (x,y)\in \mathbb{R}^{n}\times[0,\infty),\\
%\ U(x,0)=u(x), & x\in \mathbb{R}^{n}.
%\end{cases}
%$$
%Then we have
%$$(-\Delta)^{s}u(x)=-C_{n,s}\lim_{y\rightarrow0^{+}}y^{1-2s}\frac{\partial U}{\partial y},\ x\in\mathbb{R}^{n}.$$
Recall that the semi-linear Lane-Emden system in $\mathbb{R}^{n}$ takes the form
\begin{equation}\label{p1}
\begin{cases}
 -\Delta u(x)=v^{p}(x),\\
 -\Delta v(x)=u^{q}(x).
\end{cases}
\end{equation}
It is known that if $(p,q)$ are in critical or supercritical case (i.e. $\frac{1}{p+1}+\frac{1}{q+1}\leq\frac{n-2}{n}$), then system (\ref{p1}) admits some positive classical solutions on $\mathbb{R}^{n}$ (see \cite{40}).
The well-known Lane-Emden conjecture states that in the subcritical case (i.e. $\frac{1}{p+1}+\frac{1}{q+1}>\frac{n-2}{n}$), system (\ref{p1}) does not have a positive classical solution. The conjecture is known to be true for radial solutions in all dimensions (see \cite{30,33}). For non-radial solutions, and $n\leq2$, the conjecture is a consequence of known results (see \cite{31,33, Sou95}).
For $n\geq3$, when $p$ and $q$ are both subcritical case (i.e. $p,q\leq\frac{n+2}{n-2}$, but are not both equal to $\frac{n+2}{n-2}$), the conjecture was also proved (see \cite{17,36}). For $n=3$, it was also proved by Serrin and Zou \cite{39} in the subcritical case (i.e. $\frac{1}{p+1}+\frac{1}{q+1}>\frac{n-2}{n}$), but under the additional assumption that $(u,v)$ have at most polynomial growth
at $\infty$. This assumption was removed by Polacik, Quittner, and Souplet (see \cite{35, souplet}), and therefore, the conjecture was solved for $n=3$.

In the last two decades, Liouville type results for nonlinear subelliptic equations on the Heisenberg group $\mathbb{H}^n$ attracted a lot of attention. In particular,  the subelliptic Lane-Emden inequality
\begin{equation} \label{lei}
    \Delta_{\mathbb{H}^n}u+u^p\leq 0,
\end{equation}
on $\mathbb{H}^n$ was investigated by Birindelli et al. \cite{BDC} and they proved that, for $1<p\leq \frac{Q}{Q-2},$ the above inequality \eqref{lei} does not possess any  positive classical solution. Here $Q=2n+2$ is the homogeneous dimension of the Heisenberg group $\mathbb{H}^n$. Indeed, they went on to show that the exponent $\frac{Q}{Q-2}$ is optimal in the sense that, for $p>\frac{Q}{Q-2},$ the subelliptic inequality \eqref{lei} have  a nontrivial positive solution. Their results were applied to obtain a priori bounds for  subelliptic Dirichlet problems on bounded domains in $\mathbb{H}^n$ (see \cite{BDC1}). The results of \cite{BDC} were extended to the general stratified Lie groups in \cite{CC98}. After these seminal works, there are plenty of works devoted to the study of Liouville type results for the subelliptic problem on $\mathbb{H}^n:$
\begin{equation}\label{lme}
    -\Delta_{\mathbb{H}^n}u=u^p.  
\end{equation}
The Euclidean and elliptic counterpart of \eqref{lme} was studied by Gidas and Spruck \cite{GS}. They showed the nonexistence of a positive solution for the range $0<p<\frac{n+2}{n-2}.$ Later, Chen and Li \cite{CL91} simplified the proof of these results using the so-called ``the method of moving planes" and ``the Kelvin transform".   The method of moving planes has
become a powerful tool in studying qualitative properties for solutions of elliptic
equations and systems. 
The method of moving planes, which goes back to Alexandrov \cite{aa,aa1,aa2} and J. Serrin \cite{Serrin}, has been developed by many researchers later, (see \cite{GNN,zhuo,jems,cpa,chen2,chen3,am, B97, CL09, MCL11}). 

Inspired by the work of Chen and Li \cite{CL91}, the method of moving planes was introduced by Birindelli and Prajapat \cite{BP}  in the setting of the Heisenberg group $\mathbb{H}^n,$ to study Liouville theorems for positive cylindrical solutions to the semilinear subelliptic equation \eqref{lme} on $\mathbb{H}^n$ by suitably modifying the method of moving planes developed by Chen and Li \cite{CL91}.

We recall here that  function $u$ defined on $\mathbb{H}^n$ is called {\it cylindrical}  if, for any $(x,y,t)\in\mathbb{H}^n$, where
$(x,y)\in\mathbb{R}^n\times\mathbb{R}^n$ and $t\in\mathbb{R}$ is the anisotropic direction, we have $u(x,y,t)=u(r,t)$ with $r=\sqrt{|x|^2+|y|^2}$. 

Birindelli and Prajapat \cite{BP} established Gidas and Spruck \cite{GS} type results for \eqref{lme} and proved that \eqref{lme} does not have any positive cylindrical solution for $0<p <\frac{Q+2}{Q-2}.$ The sub-Laplacian $\mathbb{H}^n$  satisfies Bony's maximal principle \cite{bony} but it is not invariant under the usual hyperplane reflection as in the Euclidean case. This made authors \cite{BP} define a new kind of reflection called the  ``$H$-reflection" on the Heisenberg group with respect to the plane $T_\lambda:=\{(x,y,t) \in \mathbb{H}^n: t=\lambda\}.$ The main reason for working only with cylindrical functions in \cite{BP} is due to the fact that the $H$-reflection with respect to the plane $T_\lambda$ leaves the plane invariant but not fixed. Now, in order to settle the Gidas and Spruck conjecture, it is enough to show that any positive solution to \eqref{lei} is cylindrical. Therefore, several attempts have been made in this direction; we refer to \cite{BL03, BP01, GV01} and references therein for more details. In the aforementioned papers dealing with nonexistence results, the maximum principle is exploited along with the method of moving planes and the Kelvin transform or CR transform. We would like to also mention that many of the above results are obtained for general nonlinearity $f(u)$ with some suitable conditions on $f$ such as the Lipschitz continuity although we have discussed here some particular instances. Finally, we mention a recent advance made by Ma and Ou \cite{MO23}, in which the authors  established the Liouville theorem  for the classical solution of  \eqref{lme} for the subcritical case $1<p<\frac{Q+2}{Q-2}.$ They used a suitable generalization of the Jerison-Lee's divergence identity  and then an a priori integral estimate. 

Another very useful method to obtain Liouville type results is the combination of integral inequalities along with the method of moving planes and the Kelvin transform (see \cite{DG04}). These integral inequalities substitute the role of maximal principles in a differential form. The attribution of these integral inequalities goes to the work of Terracini \cite{Ter96, Ter96a}. The main advantage of this method is that one can handle the subcritical, critical, and supercritical cases of the Lane-Emden equation simultaneously \cite{Ter96}.  After these works, the method of moving planes combined with integral inequalities was widely used to obtain Liouville type results for elliptic equations with  general nonlinearities \cite{DG04, yu}. These methods were also adapted to study the nonexistence of solutions to elliptic systems, we cite \cite{BM02, FS05,GL08, yux, G10} and references therein for more details.

Recently, several works have been devoted to the study of the classification of  cylindrical solutions to nonlinear subelliptic  equations with general nonlinearity on the Heisenberg group $\mathbb{H}^n$: 
\begin{equation} \label{seg}
    -\Delta_{\mathbb{H}^n}u=f(u). 
\end{equation}
Indeed, Yu \cite{yujde} extended the method of integral inequalities \cite{DG04, Ter96, Ter96a} in the setting of the Heisenberg group and combined it with the method of moving planes on the Heisenberg group, to establish the Liouville type theorem for  \eqref{seg}. In \cite{Zha17}, a similar type of problem was studied for a subellipitic equation arising from the study of nonlocal equations on the Heisenberg group. For Liouville type theorems obtained using different methods such as a vector field method and test function method we refer to \cite{Xu09, KRTT} and references therein.

Inspired by the above works, it is natural to study subelliptic systems \eqref{a1} and \eqref{z1} on the Heisenberg group.  
Let us first define the notion of a weak solution in this context. 

We say that $(u,v)\in \left(H_{loc}^{1}(\mathbb{H}^{n})\cap C^{0}(\mathbb{H}^{n})\right)\times\left(H_{loc}^{1}(\mathbb{H}^{n})\cap C^{0}(\mathbb{H}^{n})\right)$ is a weak solution of system  \eqref{z1} if it satisfies
\begin{equation}\label{de}
\int_{\mathbb{H}^{n}}\nabla_{\mathbb{H}^n} u\nabla_{\mathbb{H}^n}\varphi=\int_{\mathbb{H}^{n}}f(u,v)\varphi,\quad \varphi\in C_{c}^{1}(\mathbb{H}^n),
\end{equation}
\begin{equation}\label{de1}
\int_{\mathbb{H}^{n}}\nabla_{\mathbb{H}^n} v\nabla_{\mathbb{H}^n}\varphi=\int_{\mathbb{H}^{n}}g(u,v)\varphi,\quad \varphi\in C_{c}^{1}(\mathbb{H}^n).
\end{equation}

In this paper, we are concerned with the nonexistence result of the cylindrical weak solution to  the semilinear subelliptic
systems systems \eqref{a1} and \eqref{z1} with general nonlinearity. The coupled nonlinearities in such systems make it difficult to find a starting point to apply the method of moving planes. To overcome this difficulty,  based on CR transform \cite{JL} together with the method of moving planes in $\mathbb{H}^n,$ we use integral inequalities of Terracini \cite{ Ter96, Ter96a}, which earlier proved to be helpful in \cite{yujde} for the  subellipic equation \eqref{seg}. Moreover, since the nonlinearities $f$ and $g$ are only assumed to be continuous rather than Lipschitz continuous, the classical Bony's maximal principle \cite{bony}  can not be applied. This is one of the major reasons why we work with integral inequalities 
combined with the method of moving planes.

In fact, by using the aforementioned tools, we establish  the following results in this direction.

\begin{theorem} \label{th1}
Let $n \geq 2$ and let  $(u,v)\in \left(H_{loc}^{1}(\mathbb{H}^{n})\cap C^{0}(\mathbb{H}^{n})\right)\times\left(H_{loc}^{1}(\mathbb{H}^{n})\cap C^{0}(\mathbb{H}^{n})\right)$
 be a 
positive cylindrical solution to system \eqref{a1}. Assume that $f,g:[0,+\infty)\rightarrow \mathbb{R}$ are two continuous functions satisfying the following conditions:

$(i)$ $f(t)$ and $g(t)$ are nondecreasing in $(0,+\infty)$;

$(ii)$ $h(t):=\frac{f(t)}{t^{\frac{Q+2}{Q-2}}}$ and $k(t):=\frac{g(t)}{t^{\frac{Q+2}{Q-2}}}$ are nonincreasing in $(0,+\infty)$;

$(iii)$ either $h$ or $k$ is not a constant function on $(0, \sup_{\xi \in \mathbb{H}^n} v(\xi))$ and $(0, \sup_{\xi \in \mathbb{H}^n} u(\xi)),$ respectively. 

Then   $(u,v)\equiv(C_1, C_2)$ for some constants $C_1$ and $C_2$ with $f(C_1)=0$ and $g(C_2)=0.$
\end{theorem}

Based on Theorem \ref{th1}, we state the following interesting consequence about the nonexistence of positive cylindrical solutions to the Lane-Emden system  on the Heisenberg group $\mathbb{H}^n$ in the subcritical case for any $n \geq 1$ ({\it cf.} Remark \ref{rem1}).

\begin{corollary} \label{co1}
Let $n \geq 1.$  Consider the following Lane-Emden system

\begin{equation}\label{o1}
\begin{cases}
\ -\Delta_{\mathbb{H}^{n}}u(\xi)=v^{p}(\xi) ,& \xi\in\mathbb{H}^{n},\\
\ -\Delta_{\mathbb{H}^{n}}v(\xi)=u^{q}(\xi) ,& \xi\in\mathbb{H}^{n}.
\end{cases}
\end{equation}
Then,  
 there is no positive cylindrical solution  $$(u, v)\in \left(H_{loc}^{1}(\mathbb{H}^{n})\cap C^{0}(\mathbb{H}^{n})\right)\times\left(H_{loc}^{1}(\mathbb{H}^{n})\cap C^{0}(\mathbb{H}^{n})\right)$$   
 to \eqref{o1} for\, $0<p,q<\frac{Q+2}{Q-2}$.
\end{corollary}

We note that Pohozaev and V\'eron \cite{PV00} obtained a similar result for general solution as of Corollary \ref{co1} in the case of inequality for the range $1<p,q\leq \frac{Q}{Q-2}.$ They used test function method to established these Liouville type theorems. These results were recently extended to Kirchhoff type systems on $\mathbb{H}^n$   by the third author and his collaborators \cite{KRTT} using the test function method. 

The above Theorem \ref{th1} can be extended to a more general case given by the system \eqref{z1}, that is both $f$ and $g$ depend on $(u,v)$. More precisely, we have the following  Liouville type theorem for the system \eqref{z1}.

\begin{theorem} \label{th2}
Let $n \geq 2$ and let $(u,v)\in \left(H_{loc}^{1}(\mathbb{H}^{n})\cap C^{0}(\mathbb{H}^{n})\right)\times\left(H_{loc}^{1}(\mathbb{H}^{n})\cap C^{0}(\mathbb{H}^{n})\right)$
 be a 
positive cylindrical solution to system \eqref{z1}. Assume that $f,g:[0,+\infty)\times[0,+\infty)\rightarrow \mathbb{R}$ are two continuous positive functions satisfying the following conditions:

$(i)$ $f(s,t)$ and $g(s,t)$ are nondecreasing in $t$ for fixed $s$ and, $f(s,t)$ and $g(s,t)$ are nondecreasing in $s$ for fixed $t$;

$(ii)$ there exist $p_{1} \geq 0,q_{1}>0$ with $p_{1}+q_{1}=\frac{Q+2}{Q-2}$ such that $\frac{f(s,t)}{s^{p_{1}}t^{q_{1}}}$ is nonincreasing in $t$ for fixed $s$ and $\frac{f(s,t)}{s^{p_{1}}t^{q_{1}}}$ is nonincreasing in $s$ for fixed $t$;

$(iii)$ there exist $p_{2}>0,q_{2}\geq0$ with $p_{2}+q_{2}=\frac{Q+2}{Q-2}$ such that $\frac{g(s,t)}{s^{p_{2}}t^{q_{2}}}$ is nonincreasing in $t$ for fixed $s$ and $\frac{g(s,t)}{s^{p_{2}}t^{q_{2}}}$ is nonincreasing in $s$ for fixed $t$;

$(iv)$ either $f$ or $g$ is not a constant multiple of $s^{p_1}t^{q_1}$ or $s^{p_2} t^{q_2}$, respectively. 

Then, we have $(u,v)\equiv(C_1,C_2)$ for some constants $C_1$ and $C_2$ such that $f(C_1, C_2)=0$ and $g(C_1, C_2)=0$.
\end{theorem}

%\begin{corollary} \label{co2}
%Let $(u,v)\in \left(H_{loc}^{1}(\mathbb{H}^{n})\cap C^{0}(\mathbb{H}^{n})\right)\times\left(H_{loc}^{1}(\mathbb{H}^{n})\cap C^{0}(\mathbb{H}^{n})\right)$ be a 
%nonnegative cylindrical solution to the following system
%\begin{equation}\label{o2}
%\begin{cases}
%\ -\Delta_{\mathbb{H}^{n}}u(\xi)=u^{p_{1}}(\xi)v^{q_{1}}(\xi) ,& \xi\in\mathbb{H}^{n},\\
%\ -\Delta_{\mathbb{H}^{n}}v(\xi)=u^{p_{2}}(\xi)v^{q_{2}}(\xi) ,& \xi\in\mathbb{H}^{n},
%\end{cases}
%\end{equation}
%for some $0<p_{i},q_{i}<\frac{Q+2}{Q-2}$ and $p_{i}+q_{i}=\frac{Q+2}{Q-2},i=1,2$. Then, we have  $(u,v)\equiv(0,0)$.
%\end{corollary}
\begin{remark} \label{rem1} We would like to discuss the condition  $n \geq 2$ in our main results (Theorem \ref{th1} and Theorem \ref{th2}). We need to impose this technical condition due to the unavailability of corresponding  results in the Euclidean space $\mathbb{R}^N$  for $N= 2$ (see \cite{GL08}). We emphasise here that condition $n \geq 2$ is only used in the proof of main results because of the aforementioned reason, although all the supporting Lemmata hold for all $n \geq 1.$ 
\end{remark}

The paper is organized as follows: In Section 2 we collect some well-known preliminaries related to the analysis of the Heisenberg group $\mathbb{H}^{n}$. Section 3 is devoted to the proof of Theorem \ref{th1}  by using the method of moving planes along with newly developed integral inequalities. In Section 4, we will present   the proof
of Theorem \ref{th2}.

\section{Preliminaries: the Heisenberg group}
In this section, we introduce some definitions, set up notation, and recall some basic results concerning the Heisenberg group $\mathbb{H}^{n}$. We refer to \cite{FS82, BLU07, Thang, FR16, RS19} for a complete overview of the material presented here.

The Heisenberg group $\mathbb{H}^{n}$ is $(\mathbb{R}^{2n+1}, \circ),\,\,n\geq1$, endowed with the group law $\circ$ defined by
\begin{equation}\label{f1}
\xi \circ \bar{\xi}:=\bigg(x+\bar{x},y+\bar{y},t+\bar{t}+2\sum_{i=1}^{n}(x_{i}\bar{y}_{i}-y_{i}\bar{x}_{i})\bigg),
\end{equation}
where $\xi:=(x,y,t)=(x_{1},\cdot\cdot\cdot,x_{n},y_{1},\cdot\cdot\cdot,y_{n},t)\in\mathbb{R}^{n}\times\mathbb{R}^{n}\times\mathbb{R}$ and $\bar{\xi}=(\bar{x},\bar{y},\bar{t})$.
%To denote the elements of $\mathbb{H}^{n}$ we will use the notation $(z,t)\in \mathbb{C}^{n}\times\mathbb{R}$ or $(x,y,t)\in\mathbb{R}^{n}\times\mathbb{R}^{n}\times\mathbb{R}$, where $z=x+iy$, $x=(x_{1},\cdot\cdot\cdot,x_{n})$, $y=(y_{1},\cdot\cdot\cdot,y_{n})$.

Denote by $\delta_{\tau}$ the dilations on $\mathbb{H}^{n}$ defined as
\begin{equation}\label{f2}
\delta_{\tau}(\xi)=(\tau x,\tau y,\tau^{2}t),\quad \tau>0,
\end{equation}
so that $\delta_{\tau}(\bar{\xi}\circ\xi)=\delta_{\tau}(\bar{\xi})\circ\delta_{\tau}(\xi).$

The left invariant vector fields $\{X_{1},\cdot\cdot\cdot,X_{n},Y_{1},\cdot\cdot\cdot,Y_{n},T\}$ corresponding to $\mathbb{H}^{n}$ are defined by

$$X_{i}=\frac{\partial}{\partial x_{i}}+2y_{i}\frac{\partial}{\partial t},\quad Y_{j}=\frac{\partial}{\partial y_{j}}-2x_{j}\frac{\partial}{\partial t},\quad i, j=1,2, \ldots,n, \,\,\text{and}\,\,\,\,T=\frac{\partial}{\partial t},$$
forming a basis for the Lie algebra of $\mathbb{H}^n$. It is easy to check that
$$[X_{i},Y_{j}]=-4T\delta_{ij},\ [X_{i},X_{j}]=[Y_{i},Y_{j}]=0,\quad i,j=1,\cdot\cdot\cdot,n.$$

%We note that $X_{i}$ and $Y_{i}$ ($i=1,\cdot\cdot\cdot,n$) are homogeneous of degree minus one with respect
%to $\delta_{\tau}$, namely,
%$$X_{i}(\delta_{\tau}) = \tau \delta_{\tau}(X_{i}),\quad Y_{i}(\delta_{\tau}) = \tau \delta_{\tau}(Y_{i}),$$

The horizontal gradient on $\mathbb{H}^n$ of a suitable function $\phi$ is defined as
\begin{equation}\label{f3}
\nabla_{\mathbb{H}^n}\phi=(X_{1}\phi,\ldots, X_{n}\phi,Y_{1}\phi,\ldots, Y_{n}\phi).
\end{equation}

 The sub-Laplacian $\Delta_{\mathbb{H}^n}$ on the Heisenberg group $\mathbb{H}^n$  is defined as
\begin{equation}\label{f4}
\begin{aligned}
\Delta_{\mathbb{H}^n}&:=\sum_{i=1}^{n}(X_{i}^{2}+Y_{i}^{2})\\
&=\sum_{i=1}^{n}\bigg(\frac{\partial^{2}}{\partial x_{i}^{2}}+\frac{\partial^{2}}{\partial y_{i}^{2}}+4y_{i}\frac{\partial^{2}}{\partial x_{i}\partial t}-4x_{i}(\frac{\partial^{2}}{\partial y_{i}\partial t}+4x_{i}^{2}+y_{i}^{2})\frac{\partial^{2}}{\partial t^{2}}\bigg).
\end{aligned}
\end{equation}

The family ${X_{1},\cdot\cdot\cdot,X_{n}, Y_{1},\cdot\cdot\cdot,Y_{n},}$ satisfies the H\"ormander's rank condition (see \cite{hor}), which implies that $\Delta_{\mathbb{H}^n}$ is hypoelliptic. % and the maximum principle is hold for the solutions to the equation
%including $\Delta_{\mathbb{H}^n}$ (see \cite{bony}). Then the vector field $X_{i},Y_{i},(i=1,\cdot\cdot\cdot,n)$ and their first order commutators span the whole Lie Algebra.

The even integer $Q = 2n + 2$ is called the homogeneous dimension of $\mathbb{H}^{n}$. Denote by $|\xi|_{\mathbb{H}^n}$
the (Kaplan)  distance from $\xi$ to the zero (see \cite{foll}):
\begin{equation}\label{f5}
|\xi|_{\mathbb{H}^n}:=\bigg(\sum_{i=1}^{n}(x_{i}^{2}+y_{i}^{2})^{2}+t^{2}\bigg)^{\frac{1}{4}}.
\end{equation}

%In \cite{roncal} used the norm
%$$|(z,w)|=(\sum_{i=1}^{n}(x_{i}^{2}+y_{i}^{2})^{2}+16t^{2})^{\frac{1}{4}},$$
%for $(x,y,t):= (z,w) \in \mathbb{H}^{n}$,
%which is equivalent to (\ref{c5}).

The distance between two points $\xi$ and $\eta$ in $\mathbb{H}^{n}$ is defined by
$$d_{\mathbb{H}^n}(\xi,\eta)=|\eta^{-1}\circ\xi|_{\mathbb{H}^n},$$
where $\eta^{-1}$ denotes the inverse of $\eta$ with respect to $\circ$, namely $\eta^{-1}=-\eta$.

We say that a function $u$ on $\mathbb{H}^n$ is {\it cylindrical } if for any $(x,y,t) \in \mathbb{H}^{n}$, we have $u(x, y, t) = u(r, t)$ with $r:=\sqrt{|x|^2+|y|^2}$.

It is easy to see that if $u$ is cylindrical then
$$\Delta _{\mathbb{H}^n}u(r,t)=\frac{\partial^{2}u}{\partial r^{2}}+\frac{2n-1}{r}\frac{\partial u}{\partial r}+4r^{2}\frac{\partial^{2}u}{\partial t^{2}}.$$

The open ball of radius $R > 0$ centered at $\xi$ is the set
$$  B_{\mathbb{H}^n}(\xi,R):=\{\eta\in  \mathbb{H}^{n}\mid   d_{\mathbb{H}^n}(\xi,\eta)<R\}  .$$
The map $\xi \mapsto |\xi|_{\mathbb{H}^n}$ is homogeneous of degree one with respect to  the dilations $\delta_{\tau}$ and consequently, we have
$$|B_{\mathbb{H}^n}(\xi,R)| = |B_{\mathbb{H}^n}(0,R)| = R^{Q}|B_{\mathbb{H}^n}(0,1)|,$$
where $|\cdot|$ denotes the Haar measure on $\mathbb{H}^n$.

\section{The Proof of Theorem \ref{th1}}

In this section, we outline the proof of Theorem \ref{th1} concerning the Liouville type theorem for system \eqref{a1}.  We begin with some statements required for the proof.  

Let $(u,v)$ be a pair of nonnegative continuous functions defined on $\mathbb{H}^{n}$. We introduce the {\it CR inversion} of $u$ and $v$ centered at the origin, denoted by $\bar{u}$ and $\bar{v}$, respectively as follows:
$$
\bar{u}(\xi)=\frac{1}{|\xi|_{\mathbb{H}^{n}}^{Q-2}}u(\tilde{\xi}),\,\,
\bar{v}(\xi)=\frac{1}{|\xi|_{\mathbb{H}^{n}}^{Q-2}}v(\tilde{\xi}),\,\,\,\ \xi\in\mathbb{H}^{n}\backslash \{0\},
$$
where $\xi=(x,y,t)\in \mathbb{H}^{n}\backslash\{0\}$, $\tilde{\xi}=(\tilde{x},\tilde{y},\tilde{t})\in \mathbb{H}^{n}$, and
$$
\tilde{x}_{i}=\frac{x_{i}t+y_{i}r^{2}}{{|\xi|_{\mathbb{H}}^{4}}}
,\quad \tilde{y}_{i}=\frac{y_{i}t-x_{i}r^{2}}{{|\xi|_{\mathbb{H}}^{4}}},
\quad \tilde{t}=-\frac{t}{{|\xi|_{\mathbb{H}}^{4}}},
$$
where $r^2:=(|x|^2+|y|^2).$ The CR inversion on $\mathbb{H}^n$ was first defined by Jerison and Lee \cite{JL} and serves as a suitable replacement of the {Kelvin transform} on $\mathbb{R}^n.$
Obviously, $\bar{u}$ and $ \bar{v}$ are continuous and nonnegative  on $\mathbb{H}^{n}\backslash \{0\}$. A direct computation yields (see \cite{BP, yujde} for the proof) that:

\begin{lemma} \label{lemm1}
Let $(u,v)\in \left(H_{loc}^{1}(\mathbb{H}^{n})\cap C^{0}(\mathbb{H}^{n})\right)\times\left(H_{loc}^{1}(\mathbb{H}^{n})\cap C^{0}(\mathbb{H}^{n})\right)$ be a positive weak cylindrical solution of system \eqref{a1}. Then $(\bar{u}, \bar{v})$ satisfy the following system:
\begin{equation}\label{w4}
\begin{cases}
\ -\Delta_{\mathbb{H}^{n}}\bar{u}(\xi)=\frac{1}{|\xi|_{\mathbb{H}^{n}}^{Q+2}}f(|\xi|_{\mathbb{H}^{n}}^{Q-2}\bar{v}(\xi)) ,& \xi\in\mathbb{H}^{n}\backslash \{0\},\\
\ -\Delta_{\mathbb{H}^{n}}\bar{v}(\xi)=\frac{1}{|\xi|_{\mathbb{H}^{n}}^{Q+2}}g(|\xi|_{\mathbb{H}^{n}}^{Q-2}\bar{u}(\xi)) ,& \xi\in\mathbb{H}^{n}\backslash \{0\}.
\end{cases}
\end{equation}

Moreover, $(\bar{u}, \bar{v})$  satisfy
\begin{equation}\label{32q}
    \lim_{|\xi|_{\mathbb{H}^{n}}\rightarrow \infty}|\xi|_{\mathbb{H}^{n}}^{Q-2}\bar{u}(\xi)=u(0),\ \lim_{| \xi|_{\mathbb{H}^{n}}\rightarrow \infty}|\xi|_{\mathbb{H}^{n}}^{Q-2}\bar{v}(\xi)=v(0),
\end{equation}
and therefore,  $\bar{u}, \bar{v}\in  L^{\tau+1}(\mathbb{H}^{n}\backslash B_{r}(0))\cap L^{\infty}(\mathbb{H}^{n}\backslash B_{r}(0))$ for any $r>0$, where $\tau=\frac{Q+2}{Q-2}$.
\end{lemma}

Define
\begin{equation}\label{cc}
h(t):=\frac{f(t)}{t^{\frac{Q+2}{Q-2}}},\quad k(t):=\frac{g(t)}{t^{\frac{Q+2}{Q-2}}},
\end{equation}
and substitute it in \eqref{w4} to obtain the following system:
\begin{equation}\label{cc4}
\begin{cases}
\ -\Delta_{\mathbb{H}^{n}}\bar{u}(\xi)=h(|\xi|_{\mathbb{H}^{n}}^{Q-2}\bar{v}(\xi))\bar{v}^{\frac{Q+2}{Q-2}}(\xi) ,& \xi\in\mathbb{H}^{n}\backslash \{0\},\\
\ -\Delta_{\mathbb{H}^{n}}\bar{v}(\xi)=k(|\xi|_{\mathbb{H}^{n}}^{Q-2}\bar{u}(\xi))\bar{u}^{\frac{Q+2}{Q-2}}(\xi) ,& \xi\in\mathbb{H}^{n}\backslash \{0\}.
\end{cases}
\end{equation}

Now, we are in a position to apply the moving planes method. Before that let us recall the necessary notation. 

We define the set
$$\Sigma_{\lambda}:=\{\xi:=(x, y,t )\in\mathbb{H}^{n}\mid t>\lambda\},$$
and the plane
$$T_{\lambda}:=\{\xi=(x, y,t )\in\mathbb{H}^{n}\mid t=\lambda\}.$$
For $\xi=(x,y,t)\in\mathbb{H}^{n}$, the $H$-reflection of $\xi$ with respect to the plane $T_{\lambda}$ is defined by
$$\xi_{\lambda}:=(y,x,2\lambda-t).$$

It is well-known that $-\Delta_{\mathbb{H}^n}$ is invariant under the action of $H$-reflection. This means that, if $-\Delta_{\mathbb{H}^n}u(\xi)=f(\xi)$ then $-\Delta_{\mathbb{H}^n}u(\xi_\lambda)=f(\xi_\lambda).$ 

We define the $H$-reflection $\bar{u}_{\lambda}$ of a cylindrical function $\bar{u}$ with respect to $T_{\lambda}$ by
$$\bar{u}_{\lambda}(x,y,t)=\bar{u}_{\lambda}(r,t):=\bar{u}(r,2\lambda-t)=\bar{u}(y,x,2\lambda-t).$$

Similarly, we have
$$\bar{v}_{\lambda}(x,y,t)=\bar{v}_{\lambda}(r,t):=\bar{v}(r,2\lambda-t)=\bar{v}(y,x,2\lambda-t).$$
Let $0_{\lambda}=(0,0,2\lambda)$,
then it follows from \eqref{cc4} and  the invariance
with respect to the $H$-reflection that $(\bar{u}_{\lambda},\bar{v}_{\lambda})$ satisfies
\begin{equation}\label{c5}
\begin{cases}
\ -\Delta_{\mathbb{H}^{n}}\bar{u}_{\lambda}(\xi)=h(|\xi_{\lambda}|_{\mathbb{H}^{n}}^{Q-2}\bar{v}_{\lambda}(\xi))\bar{v}_{\lambda}^{\frac{Q+2}{Q-2}}(\xi) ,& \xi\in\mathbb{H}^{n}\backslash \{0_{\lambda}\},\\
\ -\Delta_{\mathbb{H}^{n}}\bar{v}_{\lambda}(\xi)=k(|\xi_{\lambda}|_{\mathbb{H}^{n}}^{Q-2}\bar{u}_{\lambda}(\xi))\bar{u}_{\lambda}^{\frac{Q+2}{Q-2}}(\xi) ,& \xi\in\mathbb{H}^{n}\backslash \{0_{\lambda}\}.
\end{cases}
\end{equation}

In order to compare the value of $\bar{u}_{\lambda}$ and $\bar{v}_{\lambda}$ with $\bar{u}$ and $\bar{v}$, respectively,  we define
$$U_{\lambda}(\xi):=\bar{u}(\xi)-\bar{u}_{\lambda}(\xi),\quad V_{\lambda}(\xi):=\bar{v}(\xi)-\bar{v}_{\lambda}(\xi).$$
It is evident from the definition of $U_{\lambda}$ and $V_{\lambda}$ and \eqref{32q} that
\begin{equation}\label{c6}
\lim_{|\xi|_{\mathbb{H}^{n}}\rightarrow \infty} U_{\lambda}(\xi)=0,\ \lim_{|\xi|_{\mathbb{H}^{n}}\rightarrow \infty} V_{\lambda}(\xi)=0.
\end{equation}

In order to use the method of moving planes, the first step is to show that we can start the process. Namely, we first prove the following lemma.

\begin{lemma} \label{lemm2}
For any fixed $\lambda>0$, we have $\bar{u},\bar{v}\in L^{\tau+1}(\Sigma_{\lambda})\cap L^{\infty}(\Sigma_{\lambda})$, $U_{\lambda}^{+}(\xi),V_{\lambda}^{+}(\xi)\in L^{\tau+1}(\Sigma_{\lambda})\cap L^{\infty}(\Sigma_{\lambda})\cap H^{1}(\Sigma_{\lambda})$ with $\tau=\frac{Q+2}{Q-2}$, where $U_{\lambda}^{+}=\max(U_{\lambda},0)$, $V_{\lambda}^{+}=\max(V_{\lambda},0)$. Moreover, there exist $C_{\lambda}>0$ nonincreasing in $\lambda$, such that
\begin{equation}\label{w5}
\bigg(\int_{\Sigma_{\lambda}}(U_{\lambda}^{+})^{\frac{2Q}{Q-2}}d\xi\bigg)^{\frac{Q-2}{2Q}}\leq C_{\lambda}\bigg(\int_{\Sigma_{\lambda}^{v}}\frac{1}{|\xi|_{\mathbb{H}^{n}}^{2Q}}d\xi\bigg)^{\frac{2}{Q}}
\bigg(\int_{\Sigma_{\lambda}}(V_{\lambda}^{+})^{\frac{2Q}{Q-2}}d\xi\bigg)^{\frac{Q-2}{2Q}},
\end{equation}
\begin{equation}\label{w6}
\bigg(\int_{\Sigma_{\lambda}}(V_{\lambda}^{+})^{\frac{2Q}{Q-2}}d\xi\bigg)^{\frac{Q-2}{2Q}}\leq C_{\lambda}\bigg(\int_{\Sigma_{\lambda}^{u}}\frac{1}{|\xi|_{\mathbb{H}^{n}}^{2Q}}d\xi\bigg)^{\frac{2}{Q}}
\bigg(\int_{\Sigma_{\lambda}}(U_{\lambda}^{+})^{\frac{2Q}{Q-2}}d\xi\bigg)^{\frac{Q-2}{2Q}},
\end{equation}
where $\Sigma_{\lambda}^{u}=\{\xi\in\Sigma_{\lambda}\backslash \{0_{\lambda}\}\mid h(|\xi|_{\mathbb{H}^{n}}^{Q-2}\bar{u})>0,  U_{\lambda}(\xi)>0\}$ and  $\Sigma_{\lambda}^{v}=\{\xi\in\Sigma_{\lambda}\backslash \{0_{\lambda}\}\mid h(|\xi|_{\mathbb{H}^{n}}^{Q-2}\bar{v})>0,  V_{\lambda}(\xi)>0\}$. 
\end{lemma}
\begin{proof} 
We only prove \eqref{w5}, the proof of \eqref{w6} is similar. For any fixed $\lambda>0$, there exists $r>0$ such that $\Sigma_{\lambda}\subset \mathbb{H}^{n}\backslash B_{r}(0)$, then $\bar{u}$ and $U_{\lambda}^{+}\leq\bar{u}\in L^{\tau+1}(\Sigma_{\lambda})\cap L^{\infty}(\Sigma_{\lambda})$ (see Lemma \ref{lemm1}), and $\frac{1}{|\xi|_{\mathbb{H}^{n}}^{2Q}}$ is integrable in $\Sigma_{\lambda}$.

We choose a cylindrical symmetric cut-off function $0\leq\eta_{\epsilon}\leq1$ on $\mathbb{H}^n$ such that
$$
\eta_{\epsilon}(\xi)=
\begin{cases}
1, & 2\epsilon\leq|0_{\lambda}^{-1} \circ \xi|_{\mathbb{H}^{n}}\leq\frac{1}{\epsilon},\\
0 ,& |0_{\lambda}^{-1} \circ \xi|_{\mathbb{H}^{n}}\leq\epsilon\ or\ |0_{\lambda}^{-1} \circ \xi|_{\mathbb{H}^{n}}\geq\frac{2}{\epsilon},
\end{cases}
$$
and
$$
|\nabla_{\mathbb{H}^{n}}\eta_{\epsilon}|\leq
\begin{cases}
\frac{2}{\epsilon},&  \epsilon\leq|0_{\lambda}^{-1} \circ \xi|_{\mathbb{H}^{n}}\leq2\epsilon,\\
 2\epsilon ,& \frac{1}{\epsilon}\leq|0_{\lambda}^{-1} \circ \xi|_{\mathbb{H}^{n}}\leq\frac{2}{\epsilon}.
\end{cases}
$$
Now, we choose $\phi_{\epsilon}=\eta_{\epsilon}^{2}U_{\lambda}^{+}$ as a test function, then it follows from \eqref{cc4} and \eqref{c5} using the simple calculation $\nabla_{\mathbb{H}^n} \phi_\epsilon=2 (\nabla_{\mathbb{H}^n} \eta_\epsilon) \eta_\epsilon U_{\lambda}^+ +\eta_{\xi}^2 \nabla_{\mathbb{H}^n} U_{\lambda}^+ $ that
\begin{equation}\label{w7}
\begin{aligned}
\int_{\Sigma_{\lambda} \cap \left[ 2 \epsilon \leq |0_\lambda^{-1} \circ \xi |_{\mathbb{H}^n} \leq \frac{1}{\epsilon} \right]}&|\nabla_{\mathbb{H}^{n}}(U_{\lambda}^{+})|^{2}d\xi \leq 
\int_{\Sigma_{\lambda}}|\nabla_{\mathbb{H}^{n}}(U_{\lambda}^{+}\eta_{\epsilon})|^{2}d\xi \\&= \int_{\Sigma_{\lambda}} 2 \eta_\epsilon U_{\lambda}^+ \nabla_{\mathbb{H}^{n}} U_{\lambda}^{+}  \nabla_{\mathbb{H}^n} \eta_\epsilon d\xi  + \int_{\Sigma_{\lambda}} \eta_{\xi}^2 |\nabla_{\mathbb{H}^n} U_{\lambda}^+|^2 d\xi\\&\quad \quad\quad+\int_{\Sigma_{\lambda}}(U_{\lambda}^{+})^{2}|\nabla_{\mathbb{H}^{n}}\eta_{\epsilon}|^{2}d\xi \\
&=\int_{\Sigma_{\lambda}}   (\nabla_{\mathbb{H}^{n}} U_{\lambda} ) 2\eta_\epsilon U_{\xi}^+ \nabla_{\mathbb{H}^n} \eta_\epsilon  d\xi \\& \quad\quad\quad+ \int_{\Sigma_{\lambda}} \eta_{\xi}^2 |\nabla_{\mathbb{H}^n} U_{\xi}^+|^2 d\xi+\int_{\Sigma_{\lambda}}(U_{\lambda}^{+})^{2}|\nabla_{\mathbb{H}^{n}}\eta_{\epsilon}|^{2}d\xi \\
&=
\int_{\Sigma_{\lambda}}\nabla_{\mathbb{H}^{n}}U_{\lambda}\nabla_{\mathbb{H}^{n}}\phi_{\epsilon} d\xi+\int_{\Sigma_{\lambda}}(U_{\lambda}^{+})^{2}|\nabla_{\mathbb{H}^{n}}\eta_{\epsilon}|^{2}d\xi\\
&=\int_{\Sigma_{\lambda}}-\Delta_{\mathbb{H}^{n}}U_{\lambda}\phi_{\epsilon} d\xi+\int_{\Sigma_{\lambda}}(U_{\lambda}^{+})^{2}|\nabla_{\mathbb{H}^{n}}\eta_{\epsilon}|^{2}d\xi\\
&=\int_{\Sigma_{\lambda}}\bigg(h(|\xi|_{\mathbb{H}^{n}}^{Q-2}\bar{v})\bar{v}^{\frac{Q+2}{Q-2}}
-h(|\xi_{\lambda}|_{\mathbb{H}^{n}}^{Q-2}\bar{v}_{\lambda})\bar{v}_{\lambda}^{\frac{Q+2}{Q-2}}\bigg)\phi_{\epsilon} d\xi+I_{\epsilon},
\end{aligned}
\end{equation}
where $I_{\epsilon}=\int_{\Sigma_{\lambda}}(U_{\lambda}^{+})^{2}|\nabla_{\mathbb{H}^{n}}\eta_{\epsilon}|^{2}d\xi$.

Since $h$ is a nonincreasing function, $|\xi|_{\mathbb{H}^{n}}\geq|\xi_{\lambda}|_{\mathbb{H}^{n}}$ for $\xi\in\Sigma_{\lambda}$ with $\lambda>0$ we conclude that, for $\bar{v}(\xi)\geq\bar{v}_{\lambda}(\xi)\geq0$, we have
\begin{equation}\label{wv8}
    h(|\xi|_{\mathbb{H}^{n}}^{Q-2}\bar{v})\leq h(|\xi_{\lambda}|_{\mathbb{H}^{n}}^{Q-2}\bar{v}_{\lambda}).
\end{equation}
If $0 \leq \bar{v}(\xi)\leq\bar{v}_{\lambda}(\xi)$, and since $f$ is nondecreasing and $h$ is nononcreasing, we get
\begin{equation}\label{w8}
\begin{aligned}
h(|\xi|_{\mathbb{H}^{n}}^{Q-2}\bar{v})\bar{v}^{\frac{Q+2}{Q-2}}&=\frac{f(|\xi|_{\mathbb{H}^{n}}^{Q-2}\bar{v}(\xi))}{|\xi|_{\mathbb{H}^{n}}^{Q+2}}\leq\frac{f(|\xi|_{\mathbb{H}^{n}}^{Q-2}\bar{v}_{\lambda}(\xi))}{|\xi|_{\mathbb{H}^{n}}^{Q+2}}\\
&=h(|\xi|_{\mathbb{H}^{n}}^{Q-2}\bar{v}_{\lambda})\bar{v}_{\lambda}^{\frac{Q+2}{Q-2}}\leq h(|\xi_{\lambda}|_{\mathbb{H}^{n}}^{Q-2}\bar{v}_{\lambda})\bar{v}_{\lambda}^{\frac{Q+2}{Q-2}}.
\end{aligned}
\end{equation}
%Since $h$ is a nonincreasing function, $|\xi|_{\mathbb{H}^{n}}\geq |\xi_{\lambda}|_{\mathbb{H}^{n}}$  and  $\bar{v}(\xi)\geq\bar{v}_{\lambda}(\xi)$ for $\xi\in\Sigma_{\lambda}$ such that $\phi_\epsilon(\xi)>0$, we have
%\begin{equation}\label{wv8}
 %   h(|\xi|_{\mathbb{H}^{n}}^{Q-2}\bar{v})\leq h(|\xi_{\lambda}|_{\mathbb{H}^{n}}^{Q-2}\bar{v}_{\lambda}).
%\end{equation}

Therefore, using \eqref{wv8} and \eqref{w8}  in \eqref{w7} we obtain
\begin{equation}\label{w9}
\begin{aligned}
\int_{\Sigma_{\lambda} \cap \left[ 2 \epsilon \leq |0_\lambda^{-1} \circ \xi |_{\mathbb{H}^n} \leq \frac{1}{\epsilon} \right]}|\nabla_{\mathbb{H}^{n}}(U_{\lambda}^{+})|^{2}d\xi&\leq\int_{\Sigma_{\lambda}}h(|\xi|_{\mathbb{H}^{n}}^{Q-2}\bar{v})(\bar{v}^{\frac{Q+2}{Q-2}}-\bar{v}_{\lambda}^{\frac{Q+2}{Q-2}} )\phi_{\epsilon} d\xi+I_{\epsilon}\\
&=\int_{\Sigma_{\lambda}^{v}}h^+(|\xi|_{\mathbb{H}^{n}}^{Q-2}\bar{v})(\bar{v}^{\frac{Q+2}{Q-2}}-\bar{v}_{\lambda}^{\frac{Q+2}{Q-2}} ) \phi_{\epsilon} d\xi+I_{\epsilon}.
\end{aligned}
\end{equation}

Since $v$ is positive and locally bounded, there are constants $0<C_{\lambda}^{'} \leq  C_{\lambda}^{''}$ such that
\begin{equation}\label{w10}
0<C_{\lambda}^{'}:= \inf_{\xi \in \Sigma_{\lambda}}|\xi|_{\mathbb{H}^{n}}^{Q-2}\bar{v}(\xi) \leq |\xi|_{\mathbb{H}^{n}}^{Q-2}\bar{v}(\xi) \leq C_{\lambda}^{''},\quad \forall\ \xi\in \Sigma_{\lambda},
\end{equation}
and consequently, we have 
\begin{equation}\label{w11}
0\leq h^+(|\xi|_{\mathbb{H}^{n}}^{Q-2}\bar{v})\leq h^+(C_{\lambda}^{'}):=C_{\lambda},\quad \forall\ \xi\in \Sigma_{\lambda}.
\end{equation}
This shows that $C_\lambda$ is nonincreasing in $\lambda$ since $h$ is a nonincreasing function and for $\lambda_1 \leq \lambda_2,$ we can easily deduce from \eqref{w10} that $C_{\lambda_1}' \leq C_{\lambda_2}'$. Indeed,  it follows from \eqref{w10} that $C_{\lambda_1}^{'} \leq  \inf_{\xi \in \Sigma_{\lambda_1}}|\xi|_{\mathbb{H}^{n}}^{Q-2}\bar{v}(\xi)$ and the condition $\lambda_1\leq  \lambda_2$ then yields that $C_{\lambda_1}^{'} \leq  \inf_{\xi \in \Sigma_{\lambda_1}}|\xi|_{\mathbb{H}^{n}}^{Q-2}\bar{v}(\xi) \leq \inf_{\xi \in \Sigma_{\lambda_2}}|\xi|_{\mathbb{H}^{n}}^{Q-2}\bar{v}(\xi)=C_2'$ as $\Sigma_{\lambda_1} \subseteq \Sigma_{\lambda_2}.$
Moreover, for $0\leq\bar{v}_{\lambda}\leq\bar{v}$ as $\bar{v} \in L^\infty(\Sigma_{\lambda})$ for $\lambda>0$, we have using mean value theorem that
\begin{equation}\label{w12}
(\bar{v}^{\frac{Q+2}{Q-2}}-\bar{v}_{\lambda}^{\frac{Q+2}{Q-2}} )\leq \eta^{\frac{Q+2}{Q-2}-1}\frac{Q+2}{Q-2} V_{\lambda}^{+}\leq\bar{v}^{\frac{4}{Q-2}}\frac{Q+2}{Q-2} V_{\lambda}^{+}\leq
\frac{C_{\lambda}}{|\xi|_{\mathbb{H}^{n}}^{4}}V_{\lambda}^{+},
\end{equation}
where $\eta$ lies between $(\bar{v}_\lambda, \bar{v})$ 
 and  the last inequality follows from  the fact 
 that $\bar{v}$ decays at infinity as $\frac{1}{|\xi|_{\mathbb{H}^n}^{Q-2}}.$ 
Here and in the following of the paper, we always use the same $C_\lambda$ to stand for
different constants.

Combining the H\"older inequality with \eqref{w11} and \eqref{w12}, we obtain, by setting $\tau=\frac{Q+2}{Q-2},$
\begin{equation}\label{w13}
\begin{aligned}
\int_{\Sigma_{\lambda} \cap \left[ 2 \epsilon \leq |0_\lambda^{-1} \circ \xi |_{\mathbb{H}^n} \leq \frac{1}{\epsilon} \right]} &|\nabla_{\mathbb{H}^{n}}U_{\lambda}^{+}|^{2}d\xi\leq
\int_{\Sigma_{\lambda}^{v}}h^+(|\xi|_{\mathbb{H}^{n}}^{Q-2}\bar{v})(\bar{v}^{\frac{Q+2}{Q-2}}-\bar{v}_{\lambda}^{\frac{Q+2}{Q-2}} )\phi_{\epsilon} d\xi+I_{\epsilon}\\
&\leq C_{\lambda}\int_{\Sigma_{\lambda}^{v}}\frac{1}{|\xi|_{\mathbb{H}^{n}}^{4}}\eta_{\epsilon}^{2}U_{\lambda}^{+}V_{\lambda}^{+} d\xi+I_{\epsilon}\\
&\leq C_{\lambda}\left(\int_{\Sigma_{\lambda}^{v}}\frac{1}{|\xi|_{\mathbb{H}^{n}}^{2Q}} d\xi \right)^{\frac{2}{Q}}\left(\int_{\Sigma_{\lambda}}(V_{\lambda}^{+}\eta_{\epsilon})^{1+\tau}\right)^{\frac{1}{1+\tau}}\left(\int_{\Sigma_{\lambda}}(U_{\lambda}^{+}\eta_{\epsilon})^{1+\tau}\right)^{\frac{1}{1+\tau}}+I_{\epsilon}.
\end{aligned}
\end{equation}

Now, we claim that $I_{\epsilon}\rightarrow0$ as $\epsilon\rightarrow0$. To show this we define the following set 
$$D_{\epsilon}=\bigg\{\xi\in \Sigma_{\lambda} : \epsilon\leq|0_{\lambda}^{-1} \circ \xi|_{\mathbb{H}^{n}}\leq2\epsilon\ or\ \frac{1}{\epsilon}\leq|0_{\lambda}^{-1} \circ \xi|_{\mathbb{H}^{n}}\leq\frac{2}{\epsilon}\bigg\}.$$
Then, it is clear from the definition of $\eta_\epsilon$ that 
$$\int_{D_{\epsilon}}|\nabla_{\mathbb{H}^{n}}\eta_{\epsilon}|^{Q}d\xi\leq C.$$

Thus, a simple use of H\"older's inequality yields
\begin{equation}\label{w14}
\begin{aligned}
I_{\epsilon}&=\int_{\Sigma_{\lambda}^{u}}(U_{\lambda}^{+})^{2}|\nabla_{\mathbb{H}^{n}}\eta_{\epsilon}|^{2}d\xi\leq\bigg(\int_{D_{\epsilon}}(U_{\lambda}^{+})^{\tau+1}d\xi\bigg)^{\frac{2}{\tau+1}}
\bigg(\int_{D_{\epsilon}}|\nabla_{\mathbb{H}^{n}}\eta_{\epsilon}|^{Q}d\xi\bigg)^{\frac{2}{Q}}\\
&\leq C\bigg(\int_{D_{\epsilon}}(U_{\lambda}^{+})^{\tau+1}d\xi\bigg)^{\frac{2}{\tau+1}} \rightarrow0,\,\,\text{as}\,\, \epsilon\rightarrow0,
\end{aligned}
\end{equation}
provided that $U_\lambda^+ \in L^{\tau+1}(\Sigma_{\lambda}).$  This combined with  \eqref{w13} implies that $U^+ \in H^1(\Sigma_\lambda)$ as $U_{\lambda}^{+}\in L^{\tau+1}(\Sigma_{\lambda})\cap L^{\infty}(\Sigma_{\lambda})$  and $\frac{1}{|\xi|_{\mathbb{H}^{n}}^{2Q}}$ is integrable in $\Sigma_{\lambda}$

On the other hand, by Sobolev inequality (see \cite{GL92, GKR}), we have
\begin{equation}\label{w15}
\int_{\Sigma_{\lambda}}|\nabla_{\mathbb{H}^{n}}U_{\lambda}^{+}\eta_{\epsilon}|^{2}d\xi\geq C\left(\int_{\Sigma_{\lambda}}(U_{\lambda}^{+}\eta_{\epsilon})^{\frac{2Q}{Q-2}}d\xi\right)^{\frac{Q-2}{Q}}.
\end{equation}

Applying  monotone and dominated convergence theorem along with \eqref{w15} by letting $\epsilon\rightarrow0$ in \eqref{w13}, we obtain
\begin{equation*}
\begin{aligned}
\bigg(\int_{\Sigma_{\lambda}}(U_{\lambda}^{+})^{\frac{2Q}{Q-2}}d\xi\bigg)^{\frac{Q-2}{Q}}\leq 
C_{\lambda}\bigg(\int_{\Sigma_{\lambda}^{v}}\frac{1}{|\xi|_{\mathbb{H}^{n}}^{2Q}}d\xi\bigg)^{\frac{2}{Q}}
\bigg(\int_{\Sigma_{\lambda}}(U_{\lambda}^{+})^{\frac{2Q}{Q-2}}d\xi\bigg)^{\frac{Q-2}{2Q}}\bigg(\int_{\Sigma_{\lambda}}(V_{\lambda}^{+})^{\frac{2Q}{Q-2}}d\xi\bigg)^{\frac{Q-2}{2Q}}.
\end{aligned}
\end{equation*}

This completes the proof of Lemma \ref{lemm2}.
\end{proof}

\begin{lemma} \label{lemm3}
There exists $\lambda_{o}>0$ such that for all $\lambda\geq\lambda_{o}$, $U_{\lambda}(\xi)\leq0$ and
$V_{\lambda}(\xi)\leq0$ for all $\xi\in\Sigma_{\lambda}\backslash \{0_{\lambda}\}$.
\end{lemma}

\begin{proof}
Since $\frac{1}{|\xi|_{\mathbb{H}^{n}}^{2Q}}$ is integrable in $\mathbb{H}^{n}\backslash B_{r}(0)$ for any $r>0$, we have
$$\int_{\Sigma_{\lambda}^{v}}\frac{1}{|\xi|_{\mathbb{H}^{n}}^{2Q}}d\xi\leq\int_{\Sigma_{\lambda}}\frac{1}{|\xi|_{\mathbb{H}^{n}}^{2Q}}d\xi\rightarrow0,\quad as\ \lambda\rightarrow+\infty.$$
It follows that there exists $\lambda_{o}>0$ such that, for all $\lambda\geq\lambda_{o}$, we get
$$C_{\lambda}\bigg(\int_{\Sigma_{\lambda}^{u}}\frac{1}{|\xi|_{\mathbb{H}^{n}}^{2Q}}d\xi\bigg)^{\frac{1}{Q}}\bigg(\int_{\Sigma_{\lambda}^{v}}\frac{1}{|\xi|_{\mathbb{H}^{n}}^{2Q}}d\xi\bigg)^{\frac{1}{Q}}\leq\frac{1}{2}.$$
By Lemma \ref{lemm2}, we obtain that 
$$\int_{\Sigma_{\lambda}}|U_{\lambda}^{+}|^{2}d\xi=0,\ \text{and}\ \int_{\Sigma_{\lambda}}|V_{\lambda}^{+}|^{2}d\xi=0,$$
for all $\lambda\geq\lambda_{o}$, this implies that 
$U_{\lambda}(\xi)\leq0$ and
$V_{\lambda}(\xi)\leq0$ for all $\xi\in\Sigma_{\lambda}\backslash \{0_{\lambda}\}$ and $\lambda\geq\lambda_{o}$.

This completes the proof of Lemma \ref{lemm3}.
\end{proof}

Next, we can move the plane from the right to the left. More precisely, we define
$$\lambda_{1}:=\inf\{\lambda>0\mid U_{\mu}(\xi)\leq0,\ V_{\mu}(\xi)\leq0,\ \forall\,\xi\in\Sigma_{\mu}\backslash \{0_{\mu}\},\forall\, \mu\geq\lambda\}.$$ 
This is well-defined by Lemma \ref{lemm3}.
Now we have the following result.

\begin{lemma} \label{lemm4}
If $\lambda_{1}>0$, then $U_{\lambda_{1}}(\xi)\equiv0$, $V_{\lambda_{1}}(\xi)\equiv0$ for all $\xi\in \Sigma_{\lambda_{1}}\backslash \{0_{\lambda_{1}}\}$.
\end{lemma}
\begin{proof}
We deduce  from Lemma \ref{lemm3}, by using continuity of $U_\lambda$ and $V_\lambda$, that $U_{\lambda_{1}}(\xi)\leq0$, $V_{\lambda_{1}}(\xi)\leq0$ for $\xi\in \Sigma_{\lambda_{1}}\backslash \{0_{\lambda_{1}}\}$.

By \eqref{w8}, for $\xi\in \Sigma_{\lambda_{1}}\backslash \{0_{\lambda_{1}}\}$, we have 
$h(|\xi|_{\mathbb{H}^{n}}^{Q-2}\bar{v})\bar{v}^{\frac{Q+2}{Q-2}}\leq h(|\xi_{\lambda_{1}}|_{\mathbb{H}^{n}}^{Q-2}\bar{v}_{\lambda_{1}})\bar{v}_{\lambda_{1}}^{\frac{Q+2}{Q-2}}$, providing that $V_{\lambda_{1}}(\xi)\leq0$. Therefore, by \eqref{cc4} and \eqref{c5}, we get
$$
\begin{aligned}
-\Delta_{\mathbb{H}^{n}}\bar{u}\leq-\Delta_{\mathbb{H}^{n}}\bar{u}_{\lambda_{1}},
\end{aligned}
$$
which implies that $-\Delta_{\mathbb{H}^{n}}U_{\lambda_{1}}\leq0$. Since $U_{\lambda_{1}}\leq0$, by the maximum principle, either $U_{\lambda_{1}}\equiv0$ or $U_{\lambda_{1}}<0$ in $\Sigma_{\lambda_{1}}\backslash \{0_{\lambda_{1}}\}$.

Now,  let us suppose that $U_{\lambda_{1}}<0$ in $\Sigma_{\lambda_{1}}\backslash \{0_{\lambda_{1}}\}$. We note that $\frac{1}{|\xi|_{\mathbb{H}^{n}}^{2Q}}\chi_{\Sigma_{\lambda}^{v}}\rightarrow0$ pointwise  as $\lambda\rightarrow\lambda_{1}$ in $\mathbb{H}^{n}\backslash(T_{\lambda_{1}}\cup\{0_{\lambda_{1}}\})$, where $\chi_{A}$ is the characteristic function of the set $A$. 
Also, for $\lambda\in (\lambda_{1}-\delta,\lambda_{1}]$, we have $\frac{1}{|\xi|_{\mathbb{H}^{n}}^{2Q}}\chi_{\Sigma_{\lambda}^{v}}\leq\frac{1}{|\xi|_{\mathbb{H}^{n}}^{2Q}}\chi_{\Sigma_{\lambda-\delta}}\in L^{1}(\Sigma_{\lambda})$. Therefore, by dominated convergence theorem, we obtain
$$\int_{\Sigma_{\lambda}^{v}}\frac{1}{|\xi|_{\mathbb{H}^{n}}^{2Q}}d\xi\rightarrow0,\ as\ \lambda\rightarrow\lambda_{1},$$
and consequently, for $\lambda \in (\lambda_{1}-\delta,\lambda_{1})$,
$$C_{\lambda}\bigg(\int_{\Sigma_{\lambda}^{u}}\frac{1}{|\xi|_{\mathbb{H}^{n}}^{2Q}}d\xi\bigg)^{\frac{1}{Q}}\bigg(\int_{\Sigma_{\lambda}^{v}}\frac{1}{|\xi|_{\mathbb{H}^{n}}^{2Q}}d\xi\bigg)^{\frac{1}{Q}}\leq\frac{1}{2}.$$

Following the similar argument as in the proof of Lemma \ref{lemm3}, we conclude that 
$U_{\lambda}(\xi)\leq0$ and
$V_{\lambda}(\xi)\leq0$ in $\Sigma_{\lambda}\backslash \{0_{\lambda}\}$ for $\lambda\in (\lambda_{1}-\delta,\lambda_{1}]$, which contradicts with the definition of $\lambda_{1}$.

This completes the proof of Lemma \ref{lemm4}.
\end{proof}

\begin{lemma} \label{lemm5}
Assume that $u,v, f$ and $g$ are as in the statement of Theorem \ref{th1}, and suppose that $(u,v)$ is positive cylindrical on $\mathbb{H}^n$.   Then $\bar{u}, \bar{v}$ are symmetric with respect to $T_{0}$, that is, $\bar{u}$ and $\bar{v}$ are even functions in $t$-variable.
\end{lemma}

\begin{proof}

To prove that $\bar{u},\bar{v}$ are  symmetric, we still use the method of moving plane and prove the symmetry. We can carry out the procedure as above. If $\lambda_{1}>0$, then it follows from Lemma \ref{lemm4} that $\bar{u},\bar{v}$ are symmetric with respect to $T_{\lambda_{1}}$. This means that $\bar{u}=\bar{u}_{\lambda_{1}}$  and $\bar{v}=\bar{v}_{\lambda_{1}},$ which implies that $-\Delta_{\mathbb{H}^n}\bar{u}=-\Delta_{\mathbb{H}^n}\bar{u}_{\lambda_{1}}.$ This combined with $\bar{v}=\bar{v}_{\lambda_{1}},$ \eqref{cc4} and \eqref{c5} show that
\begin{equation} \label{shy}
h(|\xi_{\lambda_{1}}|_{\mathbb{H}^{n}}^{Q-2}\bar{v}_{\lambda_{1}}(\xi))=h(|\xi|_{\mathbb{H}^{n}}^{Q-2}\bar{v}(\xi)).
\end{equation}
By the assumption that $h$ is nonincreasing and the fact that $|\xi_{\lambda_{1}}|<|\xi|$ for $\xi \in \Sigma_{\lambda_{1}},$ we conclude from \eqref{shy} that $h(s)$ is constant in a left neighbourhood of $s=|\xi|_{\mathbb{H}^{n}}^{Q-2}\bar{v}(\xi)= v\left( \frac{(\tilde{x}, \tilde{y}, -t)}{|\xi|_{\mathbb{H}^{n}}^{4}} \right),\ t>\lambda_1$. In a similar manner, we can show that $h$ is constant in a right neighbourhood of $s=|\xi|_{\mathbb{H}^{n}}^{Q-2}\bar{v}(\xi)= v\left( \frac{(\tilde{x}, \tilde{y}, -t)}{|\xi|_{\mathbb{H}^{n}}^{4}} \right),\  t<\lambda_1$. In particular, this holds for $s$ close to $0$ because $s=v\left( \frac{(\tilde{x}, \tilde{y}, -t)}{|\xi|_{\mathbb{H}^{n}}^{4}} \right)$ converges to $0$ at infinity by \eqref{32q}. 
Therefore we conclude that if $\lambda_1>0,$ then $h$ is constant on $(0, \sup_{\xi \in \mathbb{H}^n} v(\xi)),$ which is a contradiction to our assumption $(iii)$. We also derive the same contradiction if $k$ is assumed  not be constant on $(0, \sup_{\xi \in \mathbb{H}^n} u(\xi)).$

If $\lambda_{1}=0$, then we conclude by continuity that $\bar{u}(\xi)\leq\bar{u}_{0}(\xi)$ and $\bar{v}(\xi)\leq\bar{v}_{0}(\xi)$ for all $\xi\in\Sigma_{0}$. In this case, we can also perform the moving plane procedure from the left and find a corresponding $\lambda'_{1}$. If $\lambda'_{1}<0$, an analogue to Lemma \ref{lemm4} shows that $\bar{u},\bar{v}$ are symmetric with respect to $T_{\lambda'_{1}}$ and we can obtain contradiction in this case as for $\lambda_1>0$ previously.   If $\lambda'_{1}=0$, then we conclude by continuity that $\bar{u}_{0}(\xi)\leq\bar{u}(\xi)$ and $\bar{v}_{0}(\xi)\leq\bar{v}(\xi)$ for all $\xi\in\Sigma_{0}$. The fact and the above inequalities imply that  $\bar{u}_{0}(\xi)=\bar{u}(\xi)$ and $\bar{v}_{0}(\xi)=\bar{v}(\xi)$ and  $\lambda_{1}=\lambda_{1}'=0.$ This shows that $\bar{u}$ and $\bar{v}$ are symmetric with respect to $T_{0}$. This completes the proof of Lemma \ref{lemm5}.
\end{proof}

Next we are ready to present the proof of Theorem \ref{th1}.
\begin{proof}
By Lemma \ref{lemm5}, we conclude that the CR inversion $\bar{u}$ and $\bar{v}$ of functions $u$ and $v$ respectively are symmetric with respect to  $T_{0},$ that is, $\bar{u}$ and $\bar{v}$ of $u$ and $v$ are even in $t$-variable. Since the choice of origin  is arbitrary in $t$-axis, then we conclude that $u$ and $v$ are independent of $t$. However, this shows that $u$ and $v$ satisfy the system
\begin{equation}\label{d11}
\begin{cases}
\ -\Delta u(\xi)=f(v(\xi)) ,& \xi\in\mathbb{R}^{2n},\\
\ -\Delta v(\xi)=g(u(\xi)) ,& \xi\in\mathbb{R}^{2n}.
\end{cases}
\end{equation}
Since $f,g$ are nondecreasing in $(0,\infty)$, and
$$\frac{f(t)}{t^{\frac{2n+2}{2n-2}}}=\frac{f(t)}{t^{\frac{Q+2}{Q-2}}}t^{\frac{Q+2}{Q-2}-\frac{2n+2}{2n-2}},$$
$$\frac{g(t)}{t^{\frac{2n+2}{2n-2}}}=\frac{g(t)}{t^{\frac{Q+2}{Q-2}}}t^{\frac{Q+2}{Q-2}-\frac{2n+2}{2n-2}},$$
are nonincreasing in $t$, it follows immediately  from   \cite[Theorem 1.1]{GL08} as $2n \geq 3$ along with the assumption $h$ or $k$ is not constant function,  that $(u,v)\equiv(C_1, C_2)$ for some constant $C_1$ and $C_2$ with $f(C_1)=0$ and $g(C_2)=0.$ 

This completes the proof of Theorem \ref{th1}.
\end{proof}

Next we provide a proof of Corollary \ref{co1}.

{\it \bf Proof of Corollary \ref{co1}:} For $n \geq 2,$ the proof is an immediate consequence of Theorem \ref{th1}. For $n=1,$ we repeat the proof of Theorem \ref{th1} including all supporting lemmata (see Remark \ref{rem1}) for the special case when $f(t)=t^p$ and $g(t)=t^q$ and use the corresponding Euclidean result from \cite{31,39, Sou95} (see also \cite{souplet}) for the dimension two instead of \cite[Theorem 1.1]{GL08} to complete the proof.   

\section{Proof of Theorem \ref{th2}}

In this section, by using the similar methods as used in Section 3,
we study the positive cylindrical (weak) solutions to the semilinear systems in the Heisenberg group
with more general nonlinearity
\begin{equation}\label{s1}
\begin{cases}
\ -\Delta_{\mathbb{H}^{n}}u(\xi)=f(u(\xi),v(\xi)) ,& \xi\in\mathbb{H}^{n},\\
\ -\Delta_{\mathbb{H}^{n}}v(\xi)=g(u(\xi),v(\xi)) ,& \xi\in\mathbb{H}^{n},
\end{cases}
\end{equation}
where $\Delta_{\mathbb{H}^{n}}$ is the  sub-Laplacian on the Heisenberg group $\mathbb{H}^{n}$.

We establish Liouville type result stated in Theorem \ref{th2} for the system (\ref{s1}). The spirit of the proofs is still the moving plane method and will be completed similarly to that of Theorem \ref{th1} with the help of some   analogue of Lemmas \ref{lemm2}-\ref{lemm4} in this more general context. Here, we carry over the notation used  in the previous Section. 

Suppose that $(u,v)\in \left(H_{loc}^{1}(\mathbb{H}^{n})\cap C^{0}(\mathbb{H}^{n})\right)\times\left(H_{loc}^{1}(\mathbb{H}^{n})\cap C^{0}(\mathbb{H}^{n})\right)$ be a weak cylindrical solution of system \eqref{s1}.
If $(u,v)$ solves (\ref{s1}), then a direct calculation yields that
\begin{equation}\label{s3}
\begin{cases}
\ -\Delta_{\mathbb{H}^{n}}\bar{u}(\xi)=\frac{1}{|\xi|_{\mathbb{H}^{n}}^{Q+2}}
f(|\xi|_{\mathbb{H}^{n}}^{Q-2}\bar{u}(\xi),|\xi|_{\mathbb{H}^{n}}^{Q-2}\bar{v}(\xi)) ,& \xi\in\mathbb{H}^{n}\backslash \{0\},\\
\ -\Delta_{\mathbb{H}^{n}}\bar{v}(\xi)=\frac{1}{|\xi|_{\mathbb{H}^{n}}^{Q+2}}
g(|\xi|_{\mathbb{H}^{n}}^{Q-2}\bar{u}(\xi),|\xi|_{\mathbb{H}^{n}}^{Q-2}\bar{v}(\xi)) ,& \xi\in\mathbb{H}^{n}\backslash \{0\},
\end{cases}
\end{equation}
and
\begin{equation}\label{ss3}
\begin{cases}
\ -\Delta_{\mathbb{H}^{n}}\bar{u}_{\lambda}(\xi)=\frac{1}{|\xi_{\lambda}|_{\mathbb{H}^{n}}^{Q+2}}
f(|\xi_{\lambda}|_{\mathbb{H}^{n}}^{Q-2}\bar{u}_{\lambda}(\xi),|\xi_{\lambda}|_{\mathbb{H}^{n}}^{Q-2}\bar{v}_{\lambda}(\xi)) ,& \xi\in\mathbb{H}^{n}\backslash \{0_{\lambda}\},\\
\ -\Delta_{\mathbb{H}^{n}}\bar{v}_{\lambda}(\xi)=\frac{1}{|\xi_{\lambda}|_{\mathbb{H}^{n}}^{Q+2}}
g(|\xi_{\lambda}|_{\mathbb{H}^{n}}^{Q-2}\bar{u}_{\lambda}(\xi),|\xi_{\lambda}|_{\mathbb{H}^{n}}^{Q-2}\bar{v}_{\lambda}(\xi)) ,& \xi\in\mathbb{H}^{n}\backslash \{0_{\lambda}\}.
\end{cases}
\end{equation}

\begin{lemma} \label{em1}
Under the assumptions of Theorem \ref{th2}, for any fixed $\lambda>0$, we have $\bar{u},\bar{v}\in L^{\tau+1}(\Sigma_{\lambda})\cap L^{\infty}(\Sigma_{\lambda})$, $U_{\lambda}^{+}(\xi),V_{\lambda}^{+}(\xi)\in L^{\tau+1}(\Sigma_{\lambda})\cap L^{\infty}(\Sigma_{\lambda})\cap H^{1}(\Sigma_{\lambda})$. Moreover, there exists $C_{\lambda}>0$, non-increasing in $\lambda$, such that the followling estimates holds:

\begin{align} \label{p4}
\bigg(\int_{\Sigma_{\lambda}}(U_{\lambda}^{+} )^{\frac{2Q}{Q-2}}d\xi\bigg)^{\frac{Q-2}{Q}}\leq& C_{\lambda}\bigg(\int_{\Sigma_{\lambda}^{u}}\frac{1}{|\xi|_{\mathbb{H}^{n}}^{2Q}}d\xi\bigg)^{\frac{2}{Q}}
\bigg(\int_{\Sigma_{\lambda}}(U_{\lambda}^{+} )^{\frac{2Q}{Q-2}}d\xi\bigg)^{\frac{Q-2}{Q}} \nonumber\\
&+C_{\lambda}\bigg(\int_{\Sigma_{\lambda}^{u}}\frac{1}{|\xi|_{\mathbb{H}^{n}}^{2Q}}d\xi\bigg)^{\frac{2}{Q}}
\bigg(\int_{\Sigma_{\lambda}}(V_{\lambda}^{+} )^{\frac{2Q}{Q-2}}d\xi\bigg)^{\frac{Q-2}{2Q}}
\bigg(\int_{\Sigma_{\lambda}}(U_{\lambda}^{+} )^{\frac{2Q}{Q-2}}d\xi\bigg)^{\frac{Q-2}{2Q}},
\end{align}
and 
\begin{align} \label{p5}
\bigg(\int_{\Sigma_{\lambda}}(V_{\lambda}^{+} )^{\frac{2Q}{Q-2}}d\xi\bigg)^{\frac{Q-2}{Q}}\leq& C_{\lambda}\bigg(\int_{\Sigma_{\lambda}^{v}}\frac{1}{|\xi|_{\mathbb{H}^{n}}^{2Q}}d\xi\bigg)^{\frac{2}{Q}}
\bigg(\int_{\Sigma_{\lambda}}(V_{\lambda}^{+} )^{\frac{2Q}{Q-2}}d\xi\bigg)^{\frac{Q-2}{Q}} \nonumber \\
&+C_{\lambda}\bigg(\int_{\Sigma_{\lambda}^{v}}\frac{1}{|\xi|_{\mathbb{H}^{n}}^{2Q}}d\xi\bigg)^{\frac{2}{Q}}
\bigg(\int_{\Sigma_{\lambda}}(U_{\lambda}^{+} )^{\frac{2Q}{Q-2}}d\xi\bigg)^{\frac{Q-2}{2Q}}
\bigg(\int_{\Sigma_{\lambda}}(V_{\lambda}^{+} )^{\frac{2Q}{Q-2}}d\xi\bigg)^{\frac{Q-2}{2Q}}.
\end{align}

\end{lemma}

\begin{proof}
We just prove \eqref{p4}, the proof of \eqref{p5} is similar. 
For any fixed $\lambda>0$, there exists $r>0$ such that $\Sigma_{\lambda}\subset \mathbb{H}^{n}\backslash B_{r}(0)$, then $\bar{u}$ and $U_{\lambda}^{+}\leq\bar{u}\in L^{\tau+1}(\Sigma_{\lambda})\cap L^{\infty}(\Sigma_{\lambda})$ and $\frac{1}{|\xi|_{\mathbb{H}^{n}}}$ is integrable in $\Sigma_{\lambda}$.

Since $\lambda>0$, it follows that $|\xi|_{\mathbb{H}^{n}}>|\xi_{\lambda}|_{\mathbb{H}^{n}}$ for all $\xi\in \Sigma_{\lambda}$, we still choose a cylindrical symmetric cut-off function $\eta_{\epsilon}$ as in Section 3.

(i) If $U_{\lambda}(\xi)\geq0$ and $V_{\lambda}(\xi)\leq0$, then by the assumptions $(i)$ in Theorem \ref{th2}, we have
\begin{equation} \label{ruz1}
    f(|\xi|_{\mathbb{H}^{n}}^{Q-2}\bar{u}(\xi),|\xi|_{\mathbb{H}^{n}}^{Q-2}\bar{v}(\xi))
\geq f\big(|\xi^{\lambda}|_{\mathbb{H}^{n}}^{Q-2}\bar{u}_{\lambda}(\xi),
|\xi_{\lambda}|_{\mathbb{H}^{n}}^{Q-2}\bar{v}(\xi)\frac{\bar{u}_{\lambda}(\xi)}{\bar{u}(\xi)}\big).
\end{equation}

By the assumptions $(ii)$ in Theorem \ref{th2}, we have
\begin{equation} \label{ruz2}
    \frac{f(|\xi_{\lambda}|_{\mathbb{H}^{n}}^{Q-2}\bar{u}_{\lambda}(\xi),
|\xi_{\lambda}|_{\mathbb{H}^{n}}^{Q-2}\bar{v}(\xi)\frac{\bar{u}_{\lambda}(\xi)}{\bar{u}(\xi)})}
{[|\xi_{\lambda}|_{\mathbb{H}^{n}}^{Q-2}\bar{u}_{\lambda}(\xi)]^{p_{1}}
[|\xi_{\lambda}|_{\mathbb{H}^{n}}^{Q-2}\bar{v}(\xi)\frac{\bar{u}_{\lambda}(\xi)}{\bar{u}(\xi)}]^{q_{1}}}\geq
\frac{f(|\xi|_{\mathbb{H}^{n}}^{Q-2}\bar{u}(\xi),|\xi|_{\mathbb{H}^{n}}^{Q-2}\bar{v}(\xi))}
{[|\xi|_{\mathbb{H}^{n}}^{Q-2}\bar{u}(\xi)]^{p_{1}}[|\xi|_{\mathbb{H}^{n}}^{Q-2}\bar{v}(\xi)]^{q_{1}}}.
\end{equation}
By using \eqref{ruz1} in \eqref{ruz2}, we deduce that
$$f(|\xi_{\lambda}|_{\mathbb{H}^{n}}^{Q-2}\bar{u}_{\lambda}(\xi),|\xi_{\lambda}|_{\mathbb{H}^{n}}^{Q-2}\bar{v}_{\lambda}(\xi))
\geq f(|\xi|_{\mathbb{H}^{n}}^{Q-2}\bar{u}(\xi),|\xi|_{\mathbb{H}^{n}}^{Q-2}\bar{v}(\xi))
\left(\frac{|\xi_{\lambda}|_{\mathbb{H}^{n}} }{|\xi|_{\mathbb{H}^{n}}}\right)^{Q+2} \Big(\frac{\bar{u}_{\lambda}(\xi)}{\bar{u}(\xi)}\Big)^{\frac{Q+2}{Q-2}},$$
which further implies that
$$
\begin{aligned}
\frac{1}{|\xi|_{\mathbb{H}^{n}}^{Q+2}}&f(|\xi|_{\mathbb{H}^{n}}^{Q-2}\bar{u}(\xi),|\xi|_{\mathbb{H}^{n}}^{Q-2}\bar{v}(\xi))
-\frac{1}{|\xi_{\lambda}|_{\mathbb{H}^{n}}^{Q+2}}f(|\xi_{\lambda}|_{\mathbb{H}^{n}}^{Q-2}\bar{u}_{\lambda}
(\xi),|\xi_{\lambda}|_{\mathbb{H}^{n}}^{Q-2}\bar{v}_{\lambda}(\xi))\\
&\leq\frac{1}{|\xi|_{\mathbb{H}^{n}}^{Q+2}}f(|\xi|_{\mathbb{H}^{n}}^{Q-2}\bar{u}(\xi),|\xi|_{\mathbb{H}^{n}}^{Q-2}\bar{v}(\xi))
\bigg(1-\Big(\frac{\bar{u}_{\lambda}(\xi)}{\bar{u}(\xi)}\Big)^{\frac{Q+2}{Q-2}}\bigg)\\
&\leq\frac{1}{|\xi|_{\mathbb{H}^{n}}^{Q+2}}f(|\xi|_{\mathbb{H}^{n}}^{Q-2}\bar{u}(\xi),|\xi|_{\mathbb{H}^{n}}^{Q-2}\bar{v}(\xi))\frac{Q+2}{Q-2}
\bigg(1-\frac{\bar{u}_{\lambda}(\xi)}{\bar{u}(\xi)}\bigg)\\
&\leq\frac{C}{|\xi|_{\mathbb{H}^{n}}^{Q+2}}\frac{f(|\xi|_{\mathbb{H}^{n}}^{Q-2}\bar{u}(\xi),|\xi|_{\mathbb{H}^{n}}^{Q-2}\bar{v}(\xi))}{\bar{u}(\xi)}
(\bar{u}(\xi)-\bar{u}_{\lambda}(\xi))\\
&\leq\frac{C}{|\xi|_{\mathbb{H}^{n}}^{4}}\frac{f(|\xi|_{\mathbb{H}^{n}}^{Q-2}\bar{u}(\xi),|\xi|_{\mathbb{H}^{n}}^{Q-2}\bar{v}(\xi))}
{|\xi|_{\mathbb{H}^{n}}^{Q-2}\bar{u}(\xi)}
(\bar{u}(\xi)-\bar{u}_{\lambda}(\xi))
\\
&\leq\frac{C}{|\xi|_{\mathbb{H}^{n}}^{4}}\frac{f(|\xi|_{\mathbb{H}^{n}}^{Q-2}\bar{u}(\xi),|\xi|_{\mathbb{H}^{n}}^{Q-2}\bar{v}(\xi))}
{|\xi|_{\mathbb{H}^{n}}^{Q-2}\bar{u}(\xi) (|\xi|_{\mathbb{H}^{n}}^{Q-2}\bar{v}(\xi))^{\frac{4}{Q-2}}} (|\xi|_{\mathbb{H}^{n}}^{Q-2}\bar{v}(\xi))^{\frac{4}{Q-2}}
(\bar{u}(\xi)-\bar{u}_{\lambda}(\xi))\\
&\leq C\frac{f(|\xi|_{\mathbb{H}^{n}}^{Q-2}\bar{u}(\xi),|\xi|_{\mathbb{H}^{n}}^{Q-2}\bar{v}(\xi))}
{|\xi|_{\mathbb{H}^{n}}^{Q-2}\bar{u}(\xi) (|\xi|_{\mathbb{H}^{n}}^{Q-2}\bar{v}(\xi))^{\frac{4}{Q-2}}} (\bar{v}(\xi))^{\frac{4}{Q-2}}
(\bar{u}(\xi)-\bar{u}_{\lambda}(\xi))\\
&\leq\frac{C_{\lambda}}{|\xi|_{\mathbb{H}^{n}}^{4}}(\bar{u}(\xi)-\bar{u}_{\lambda}(\xi)),
\end{aligned}
$$
where we have used   the mean value theorem with the observation that $\frac{\bar{u}_{\lambda}(\xi)}{\bar{u}(\xi)} \leq 1$ in the second inequality and the fact that $f$ is a positive continuous function such that $\frac{f(s, t)}{s\,\, t^{\frac{4}{Q-2}}}$ is nonincreasing in $s$ (assumption $(ii)$) along with  the $\bar{v}(\xi)$ decay as $\frac{1}{|\xi|_{\mathbb{H}^{n}}^{Q-2}}$  to deduce the last inequality.

$(ii)$ If $U_{\lambda}(\xi)\geq0$ and $V_{\lambda}(\xi)>0$, then by arguing similar to $(i)$ we have
$$\frac{f(|\xi_{\lambda}|_{\mathbb{H}^{n}}^{Q-2}\bar{u}_{\lambda}(\xi),|\xi_{\lambda}|_{\mathbb{H}^{n}}^{Q-2}\bar{v}_{\lambda}(\xi))}
{|\xi_{\lambda}|_{\mathbb{H}^{n}}^{Q+2}\bar{u}_{\lambda}^{p_{1}}(\xi)\bar{v}_{\lambda}^{q_{1}}(\xi)}
\geq\frac{f(|\xi|_{\mathbb{H}^{n}}^{Q-2}\bar{u}(\xi),|\xi|_{\mathbb{H}^{n}}^{Q-2}\bar{v}(\xi))}
{|\xi|_{\mathbb{H}^{n}}^{Q+2}\bar{u}^{p_{1}}(\xi)\bar{v}^{q_{1}}(\xi)},$$
that is
$$\frac{f(|\xi_{\lambda}|_{\mathbb{H}^{n}}^{Q-2}\bar{u}_{\lambda}(\xi),|\xi^{\lambda}|_{\mathbb{H}^{n}}^{Q-2}\bar{v}_{\lambda}(\xi))}
{|\xi_{\lambda}|_{\mathbb{H}^{n}}^{Q+2}}
\geq\frac{f(|\xi|_{\mathbb{H}^{n}}^{Q-2}\bar{u}(\xi),|\xi|_{\mathbb{H}^{n}}^{Q-2}\bar{v}(\xi))}
{|\xi|_{\mathbb{H}^{n}}^{Q+2}}\bigg(\frac{\bar{u}_{\lambda}(\xi)}{\bar{u}(\xi)}\bigg)^{p_{1}}\bigg(\frac{\bar{v}_{\lambda}(\xi)}{\bar{v}(\xi)}\bigg)^{q_{1}}.$$
So we have
$$
\begin{aligned}
\frac{f(|\xi|_{\mathbb{H}^{n}}^{Q-2}\bar{u}(\xi),|\xi|_{\mathbb{H}^{n}}^{Q-2}\bar{v}(\xi))}
{|\xi|_{\mathbb{H}^{n}}^{Q+2}}-&\frac{f(|\xi_{\lambda}|_{\mathbb{H}^{n}}^{Q-2}
\bar{u}_{\lambda}(\xi),|\xi_{\lambda}|_{\mathbb{H}^{n}}^{Q-2}\bar{v}_{\lambda}(\xi))}
{|\xi_{\lambda}|_{\mathbb{H}^{n}}^{Q+2}}\\
\leq&\frac{f(|\xi|_{\mathbb{H}^{n}}^{Q-2}\bar{u}(\xi),|\xi|_{\mathbb{H}^{n}}^{Q-2}\bar{v}(\xi))}
{|\xi|_{\mathbb{H}^{n}}^{Q+2}}\bigg(1-\Big(\frac{\bar{u}_{\lambda}(\xi)}{\bar{u}(\xi)}\Big)^{p_{1}}\Big(\frac{\bar{v}_{\lambda}(\xi)}{\bar{v}(\xi)}\Big)^{q_{1}}\bigg)\\
\leq&\frac{f(|\xi|_{\mathbb{H}^{n}}^{Q-2}\bar{u}(\xi),|\xi|_{\mathbb{H}^{n}}^{Q-2}\bar{v}(\xi))}
{|\xi|_{\mathbb{H}^{n}}^{Q+2}}
\bigg(1-\Big(\frac{\bar{u}_{\lambda}(\xi)}{\bar{u}(\xi)}\Big)^{\frac{Q+2}{Q-2}}\Big(\frac{\bar{v}_{\lambda}(\xi)}{\bar{v}(\xi)}\Big)^{\frac{Q+2}{Q-2}}\bigg)\\
\leq&\frac{f(|\xi|_{\mathbb{H}^{n}}^{Q-2}\bar{u}(\xi),|\xi|_{\mathbb{H}^{n}}^{Q-2}\bar{v}(\xi))}
{|\xi|_{\mathbb{H}^{n}}^{Q+2}}\frac{Q+2}{Q-2}
\bigg(\Big(1-\frac{\bar{u}_{\lambda}(\xi)}{\bar{u}(\xi)}\Big)+\Big(1-\frac{\bar{v}_{\lambda}(\xi)}{\bar{v}(\xi)}\Big)\bigg)\\
\leq&\frac{C}{|\xi|_{\mathbb{H}^{n}}^{Q+2}}
\bigg(\frac{f(|\xi|_{\mathbb{H}^{n}}^{Q-2}\bar{u}(\xi),|\xi|_{\mathbb{H}^{n}}^{Q-2}\bar{v}(\xi))}{\bar{u}(\xi)}\big(\bar{u}(\xi)-\bar{u}_{\lambda}(\xi)\big)\\&+
\frac{f(|\xi|_{\mathbb{H}^{n}}^{Q-2}\bar{u}(\xi),|\xi|_{\mathbb{H}^{n}}^{Q-2}\bar{v}(\xi))}{\bar{v}(\xi)}\big(\bar{v}(\xi)-\bar{v}_{\lambda}(\xi)\big)\bigg)\\
\leq&\frac{C_{\lambda}}{|\xi|_{\mathbb{H}^{n}}^{4}}\bigg(\big(\bar{u}(\xi)-\bar{u}_{\lambda}(\xi)\big)+\big(\bar{v}(\xi)-\bar{v}_{\lambda}(\xi)\big)\bigg).
\end{aligned}
$$

$(iii)$ If $U_{\lambda}(\xi)<0$ and $V_{\lambda}(\xi)\geq0$, then we have change the role of $\bar{u},\bar{v}$ in case $(i)$ and obtain
$$
\begin{aligned}
\frac{1}{|\xi|_{\mathbb{H}^{n}}^{Q+2}}f(|\xi|_{\mathbb{H}^{n}}^{Q-2}\bar{u}(\xi),|\xi|_{\mathbb{H}^{n}}^{Q-2}\bar{v}(\xi))
&-\frac{1}{|\xi_{\lambda}|_{\mathbb{H}^{n}}^{Q+2}}f(|\xi_{\lambda}|_{\mathbb{H}^{n}}^{Q-2}\bar{u}_{\lambda}
(\xi),|\xi^{\lambda}|_{\mathbb{H}^{n}}^{Q-2}\bar{v}_{\lambda}(\xi))\\
&\leq\frac{C_{\lambda}}{|\xi|_{\mathbb{H}^{n}}^{4}}(\bar{v}(\xi)-\bar{v}_{\lambda}(\xi)).
\end{aligned}
$$

$(iv)$ If $U_{\lambda}(\xi)<0$ and $V_{\lambda}(\xi)<0$, then we have
$$
\begin{aligned}
\frac{f(|\xi|_{\mathbb{H}^{n}}^{Q-2}\bar{u}(\xi),|\xi|_{\mathbb{H}^{n}}^{Q-2}\bar{v}(\xi))}
{|\xi|_{\mathbb{H}^{n}}^{Q+2}}&\leq\frac{f(|\xi_{\lambda}|_{\mathbb{H}^{n}}^{Q-2}
\bar{u}_{\lambda}(\xi),|\xi_{\lambda}|_{\mathbb{H}^{n}}^{Q-2}\bar{v}_{\lambda}(\xi))}
{|\xi^{\lambda}|_{\mathbb{H}^{n}}^{Q+2}}\\
&=\frac{f(|\xi|_{\mathbb{H}^{n}}^{Q-2}\bar{u}(\xi),|\xi|_{\mathbb{H}^{n}}^{Q-2}\bar{v}(\xi))}
{[|\xi|_{\mathbb{H}^{n}}^{Q-2}\bar{u}_{\lambda}(\xi)]^{p_{1}}[|\xi|_{\mathbb{H}^{n}}^{Q-2}\bar{v}_{\lambda}(\xi)]^{q_{1}}}
\bar{u}_{\lambda}^{p_{1}}(\xi)\bar{v}_{\lambda}^{q_{1}}(\xi)\\
&\leq\frac{f(|\xi_{\lambda}|_{\mathbb{H}^{n}}^{Q-2}\bar{u}(\xi),|\xi_{\lambda}|_{\mathbb{H}^{n}}^{Q-2}\bar{v}(\xi))}
{[|\xi_{\lambda}|_{\mathbb{H}^{n}}^{Q-2}\bar{u}_{\lambda}(\xi)]^{p_{1}}[|\xi_{\lambda}|_{\mathbb{H}^{n}}^{Q-2}\bar{v}_{\lambda}(\xi)]^{q_{1}}}
\bar{u}_{\lambda}^{p_{1}}(\xi)\bar{v}_{\lambda}^{q_{1}}(\xi)\\
&\leq\frac{f(|\xi_{\lambda}|_{\mathbb{H}^{n}}^{Q-2}\bar{u}_{\lambda}(\xi),|\xi_{\lambda}|_{\mathbb{H}^{n}}^{Q-2}\bar{v}_{\lambda}(\xi))}
{|\xi_{\lambda}|_{\mathbb{H}^{n}}^{Q+2}}.
\end{aligned}
$$
The rest of calculation follows exactly similar to case $(ii),$ therefore we skip it.

Therefore,  we deduce from \eqref{s3} and \eqref{ss3} by using cases $(i)$-$(iv)$ that
\begin{equation}\label{i1}
-\Delta_{\mathbb{H}^{n}}U_{\lambda}\leq \frac{C_{\lambda}}{|\xi|_{\mathbb{H}^{n}}^{4}}\bigg(\big(\bar{u}(\xi)-\bar{u}_{\lambda}(\xi)\big)^{+}+\big(\bar{v}(\xi)-\bar{v}_{\lambda}(\xi)\big)^{+}\bigg).
\end{equation}
Similarly, we obtain
\begin{equation}\label{i2}
-\Delta_{\mathbb{H}^{n}}V_{\lambda}\leq \frac{C_{\lambda}}{|\xi|_{\mathbb{H}^{n}}^{4}}\bigg(\big(\bar{u}(\xi)-\bar{u}_{\lambda}(\xi)\big)^{+}+\big(\bar{v}(\xi)-\bar{v}_{\lambda}(\xi)\big)^{+}\bigg).
\end{equation}
Hence,  after a calculation similar to that of \eqref{w7}, we conclude that

\begin{equation}\label{z5}
\begin{aligned}
\int_{\Sigma_{\lambda}}|\nabla_{\mathbb{H}^{n}}(U_{\lambda}^{+}\eta_{\epsilon})|^{2}d\xi
=&\int_{\Sigma_{\lambda}}-\Delta_{\mathbb{H}^{n}}(U_{\lambda}^{+}\eta_{\epsilon}^{2}U_{\lambda}^{+} )d\xi+I_{\epsilon}\\
\leq&\int_{\Sigma_{\lambda}}\frac{C_{\lambda}}{|\xi|_{\mathbb{H}^{n}}^{4}}(U_{\lambda}^{+}+V_{\lambda}^{+})\eta_{\epsilon}^{2}U_{\lambda}^{+} d\xi
+I_{\epsilon},
\end{aligned}
\end{equation}
where $I_{\epsilon}=\int_{\Sigma_{\lambda}^{u}}(U_{\lambda}^{+})^{2}|\nabla_{\mathbb{H}^{n}}\eta_{\epsilon}|^{2}d\xi$. We can also prove that $I_{\epsilon}\rightarrow0$ as $\epsilon\rightarrow0$ as in \eqref{w14}.

By the H\"older inequality, we obtain

\begin{equation}\label{z6}
\begin{aligned}
\int_{\Sigma_{\lambda}}\frac{1}{|\xi|_{\mathbb{H}^{n}}^{4}}(U_{\lambda}^{+}+V_{\lambda}^{+})&\eta_{\epsilon}^{2}U_{\lambda}^{+} d\xi\leq C_{\lambda}\bigg(\int_{\Sigma_{\lambda}^{u}}\frac{1}{|\xi|_{\mathbb{H}^{n}}^{2Q}}d\xi\bigg)^{\frac{2}{Q}}
\bigg(\int_{\Sigma_{\lambda}}(U_{\lambda}^{+} \eta_{\epsilon})^{\frac{2Q}{Q-2}}d\xi\bigg)^{\frac{Q-2}{Q}}\\
&+C_{\lambda}\bigg(\int_{\Sigma_{\lambda}^{u}}\frac{1}{|\xi|_{\mathbb{H}^{n}}^{2Q}}d\xi\bigg)^{\frac{2}{Q}}
(\int_{\Sigma_{\lambda}}(V_{\lambda}^{+} \eta_{\epsilon})^{\frac{2Q}{Q-2}}d\xi\big)^{\frac{Q-2}{2Q}}
\bigg(\int_{\Sigma_{\lambda}}(U_{\lambda}^{+} \eta_{\epsilon})^{\frac{2Q}{Q-2}}d\xi\bigg)^{\frac{Q-2}{2Q}}.
\end{aligned}
\end{equation}

Now, combining \eqref{z5} with \eqref{z6} and letting $\epsilon\rightarrow0$ in \eqref{z5}, we get
\begin{equation}\label{i3}
\begin{aligned}
\int_{\Sigma_{\lambda}}|\nabla_{\mathbb{H}^{n}}U_{\lambda}^{+}|^{2}d\xi\leq& C_{\lambda}\bigg(\int_{\Sigma_{\lambda}^{u}}\frac{1}{|\xi|_{\mathbb{H}^{n}}^{2Q}}d\xi\bigg)^{\frac{2}{Q}}
\bigg(\int_{\Sigma_{\lambda}}(U_{\lambda}^{+} )^{\frac{2Q}{Q-2}}d\xi\bigg)^{\frac{Q-2}{Q}}\\
&+C_{\lambda}\bigg(\int_{\Sigma_{\lambda}^{u}}\frac{1}{|\xi|_{\mathbb{H}^{n}}^{2Q}}d\xi\bigg)^{\frac{2}{Q}}
(\int_{\Sigma_{\lambda}}(V_{\lambda}^{+} )^{\frac{2Q}{Q-2}}d\xi\big)^{\frac{Q-2}{2Q}}
\bigg(\int_{\Sigma_{\lambda}}(U_{\lambda}^{+} )^{\frac{2Q}{Q-2}}d\xi\bigg)^{\frac{Q-2}{2Q}}.
\end{aligned}
\end{equation}

Similarly, we have
\begin{equation}\label{i4}
\begin{aligned}
\int_{\Sigma_{\lambda}}|\nabla_{\mathbb{H}^{n}}V_{\lambda}^{+}|^{2}d\xi\leq& C_{\lambda}\bigg(\int_{\Sigma_{\lambda}^{v}}\frac{1}{|\xi|_{\mathbb{H}^{n}}^{2Q}}d\xi\bigg)^{\frac{2}{Q}}
\bigg(\int_{\Sigma_{\lambda}}(V_{\lambda}^{+} )^{\frac{2Q}{Q-2}}d\xi\bigg)^{\frac{Q-2}{Q}}\\
&+C_{\lambda}\bigg(\int_{\Sigma_{\lambda}^{v}}\frac{1}{|\xi|_{\mathbb{H}^{n}}^{2Q}}d\xi\bigg)^{\frac{2}{Q}}
(\int_{\Sigma_{\lambda}}(U_{\lambda}^{+} )^{\frac{2Q}{Q-2}}d\xi\big)^{\frac{Q-2}{2Q}}
\bigg(\int_{\Sigma_{\lambda}}(V_{\lambda}^{+} )^{\frac{2Q}{Q-2}}d\xi\bigg)^{\frac{Q-2}{2Q}}.
\end{aligned}
\end{equation}

This completes the proof of Lemma \ref{em1}.
\end{proof}

\begin{lemma} \label{em2}
Under the assumptions of Theorem \ref{th2}, there exists $\lambda_{o}>0$ such that for all $\lambda\geq\lambda_{o}$, $U_{\lambda}\leq0$ and $V_{\lambda}\leq0$ for all $\xi\in \Sigma_{\lambda}\backslash \{0_{\lambda}\}$.
\end{lemma}
\begin{proof}
  This proof of Lemma \ref{em2} is similar to Lemma \ref{lemm3}, therefore, we skip it.  
\end{proof}

Now we can move the plane from the right to the left. More precisely, we define
$$\lambda_{1}=\inf\{\lambda>0\mid U_{\mu}(\xi)\leq0,\ V_{\mu}(\xi)\leq0,\ \xi\in\Sigma_{\mu}\backslash \{0_{\mu}\},\ \mu\geq\lambda\},$$
then we have the following lemma.

\begin{lemma} \label{le2}
If $\lambda_{1}>0$, then $U_{\lambda_{1}}(\xi)\equiv0$ and $V_{\lambda_{1}}(\xi)\equiv0$ for all $\xi\in \Sigma_{\lambda_{1}}\backslash \{0_{\lambda_{1}}\}$.
\end{lemma}

\begin{proof}
By continuity, we deduce that $U_{\lambda_{1}}(\xi)\leq0$, $V_{\lambda_{1}}(\xi)\leq0$ for $\xi\in \Sigma_{\lambda_{1}}\backslash \{0_{\lambda_{1}}\}$.
Next, we will prove that if $U_{\lambda_{1}}=0$ at some point $\hat{\xi}\in \Sigma_{\lambda_{1}}\backslash \{0_{\lambda_{1}}\}$, then in a neighborhood of $\hat{\xi}$, $U_{\lambda_{1}}=0$, and hence by continuity, $U_{\lambda_{1}}\equiv0$ in $\Sigma_{\lambda_{1}}\backslash \{0_{\lambda_{1}}\}$.

Since $U_{\lambda_{1}}(\hat{\xi})=0$, we have $|\xi|_{\mathbb{H}^{n}}^{Q+2}\bar{u}(\hat{\xi})>|\xi_{\lambda_{1}}|_{\mathbb{H}^{n}}^{Q+2}\bar{u}_{\lambda_{1}}(\hat{\xi})$ for $\xi\in\Omega$ a neighborhood of $\hat{\xi}$. By using the same arguments as we used in the proof of Lemma \ref{em1}, we have 

\begin{equation}\label{r1}
\begin{cases}
-\Delta_{\mathbb{H}^{n}}U_{\lambda_{1}}+C U_{\lambda_{1}}\leq0,& in\ \Omega,\\
U_{\lambda_{1}}\leq0,\ U_{\lambda_{1}}(\hat{\xi})=0,& in\ \Omega,
\end{cases}
\end{equation}
where $C>0$ is a constant depending on $\hat{\xi}$, $\Omega$ is a neighborhood of $\hat{\xi}$. Hence, by the maximal principle, $U_{\lambda_{1}}\equiv0$, in $\Omega$. 

Now we claim that $U_{\lambda_{1}}\equiv0$ implies $V_{\lambda_{1}}\equiv0$. In fact, by \eqref{s3} and \eqref{ss3}, we obtain
$$
\begin{aligned}
\frac{f(|\xi|_{\mathbb{H}^{n}}^{Q-2}\bar{u}(\xi),|\xi|_{\mathbb{H}^{n}}^{Q-2}\bar{v}(\xi))}
{|\xi|_{\mathbb{H}^{n}}^{Q+2}}=&\frac{f(|\xi_{\lambda_{1}}|_{\mathbb{H}^{n}}^{Q-2}
\bar{u}_{\lambda_{1}}(\xi),|\xi_{\lambda_{1}}|_{\mathbb{H}^{n}}^{Q-2}\bar{v}_{\lambda_{1}}(\xi))}
{|\xi_{\lambda_{1}}|_{\mathbb{H}^{n}}^{Q+2}}\\
\geq&\frac{f(|\xi|_{\mathbb{H}^{n}}^{Q-2}
\bar{u}(\xi),|\xi_{\lambda_{1}}|_{\mathbb{H}^{n}}^{Q-2}\bar{v}_{\lambda_{1}}(\xi))}
{|\xi_{\lambda_{1}}|_{\mathbb{H}^{n}}^{Q+2}}\\
>&\frac{f(|\xi|_{\mathbb{H}^{n}}^{Q-2}
\bar{u}(\xi),|\xi_{\lambda_{1}}|_{\mathbb{H}^{n}}^{Q-2}\bar{v}_{\lambda_{1}}(\xi))}
{|\xi|_{\mathbb{H}^{n}}^{Q+2}}.
\end{aligned}
$$

Since $f(s,t)$ is nondecreasing in $t$, we deduce from the above inequality that
\begin{equation}\label{r2}
|\xi|_{\mathbb{H}^{n}}^{Q-2}\bar{v}(\xi)>|\xi_{\lambda_{1}}|_{\mathbb{H}^{n}}^{Q-2}\bar{v}_{\lambda_{1}}(\xi).
\end{equation}

On the other hand,
$$
\begin{aligned}
\frac{f(|\xi|_{\mathbb{H}^{n}}^{Q-2}\bar{u}(\xi),|\xi|_{\mathbb{H}^{n}}^{Q-2}\bar{v}(\xi))}
{|\xi|_{\mathbb{H}^{n}}^{Q+2}}=&\frac{f(|\xi_{\lambda_{1}}|_{\mathbb{H}^{n}}^{Q-2}
\bar{u}_{\lambda_{1}}(\xi),|\xi_{\lambda_{1}}|_{\mathbb{H}^{n}}^{Q-2}\bar{v}_{\lambda_{1}}(\xi))}
{|\xi_{\lambda_{1}}|_{\mathbb{H}^{n}}^{Q+2}}\\
\geq&\frac{f(|\xi_{\lambda_{1}}|_{\mathbb{H}^{n}}^{Q-2}
\bar{u}(\xi),|\xi_{\lambda_{1}}|_{\mathbb{H}^{n}}^{Q-2}\bar{v}(\xi))}
{|\xi_{\lambda_{1}}|_{\mathbb{H}^{n}}^{Q+2}}\\
\geq&\frac{f(|\xi|_{\mathbb{H}^{n}}^{Q-2}
\bar{u}(\xi),|\xi|_{\mathbb{H}^{n}}^{Q-2}\bar{v}(\xi))}
{|\xi|_{\mathbb{H}^{n}}^{Q+2}}.
\end{aligned}
$$
Hence
\begin{equation}\label{zr}
\begin{aligned}
\frac{f(|\xi|_{\mathbb{H}^{n}}^{Q-2}\bar{u}(\xi),|\xi|_{\mathbb{H}^{n}}^{Q-2}\bar{v}(\xi))}
{|\xi|_{\mathbb{H}^{n}}^{Q+2}}=\frac{f(|\xi_{\lambda_{1}}|_{\mathbb{H}^{n}}^{Q-2}
\bar{u}(\xi),|\xi_{\lambda_{1}}|_{\mathbb{H}^{n}}^{Q-2}\bar{v}(\xi))}
{|\xi_{\lambda_{1}}|_{\mathbb{H}^{n}}^{Q+2}},
\end{aligned}
\end{equation}
therefore
\begin{equation}\label{i6}
\frac{f(|\xi|_{\mathbb{H}^{n}}^{Q-2}\bar{u}(\xi),|\xi|_{\mathbb{H}^{n}}^{Q-2}\bar{v}(\xi))}
{(|\xi|_{\mathbb{H}^{n}}^{Q-2}\bar{u}(\xi))^{p_{1}}(|\xi|_{\mathbb{H}^{n}}^{Q-2}\bar{v}(\xi))^{q_{1}}}=\frac{f(|\xi_{\lambda_{1}}|_{\mathbb{H}^{n}}^{Q-2}
\bar{u}(\xi),|\xi_{\lambda_{1}}|_{\mathbb{H}^{n}}^{Q-2}\bar{v}(\xi))}
{(|\xi_{\lambda_{1}}|_{\mathbb{H}^{n}}^{Q-2}\bar{u}(\xi))^{p_{1}}(|\xi_{\lambda_{1}}|_{\mathbb{H}^{n}}^{Q-2}\bar{v}(\xi))^{q_{1}}}.
\end{equation}
 
By \eqref{r2}, \eqref{i6} and assumption $(i)$ of Theorem \ref{th2}, we obtain
\begin{equation}\label{i7}
\frac{f(|\xi|_{\mathbb{H}^{n}}^{Q-2}\bar{u}(\xi),|\xi|_{\mathbb{H}^{n}}^{Q-2}\bar{v}(\xi))}
{(|\xi|_{\mathbb{H}^{n}}^{Q-2}\bar{u}(\xi))^{p_{1}}(|\xi|_{\mathbb{H}^{n}}^{Q-2}\bar{v}(\xi))^{q_{1}}}=\frac{f(|\xi_{\lambda_{1}}|_{\mathbb{H}^{n}}^{Q-2}
\bar{u}(\xi),|\xi_{\lambda_{1}}|_{\mathbb{H}^{n}}^{Q-2}\bar{v}_{\lambda_{1}}(\xi))}
{(|\xi_{\lambda_{1}}|_{\mathbb{H}^{n}}^{Q-2}\bar{u}(\xi))^{p_{1}}(|\xi_{\lambda_{1}}|_{\mathbb{H}^{n}}^{Q-2}\bar{v}_{\lambda_{1}}(\xi))^{q_{1}}},
\end{equation}
and consequently, it follows from \eqref{zr} and \eqref{i7} that $\bar{v}^{q_{1}}=\bar{v}_{\lambda_{1}}^{q_{1}}$ and therefore, $V_{\lambda_{1}}\equiv0$ as $q_{1}>0$. 

Suppose that $U_{\lambda_{1}}\not\equiv0$ and $V_{\lambda_{1}}\not\equiv0$ in $\Sigma_{\lambda_{1}}\backslash \{0_{\lambda_{1}}\}$, then $U_{\lambda_{1}}<0$, $V_{\lambda_{1}}<0$ in $\Sigma_{\lambda_{1}}\backslash \{0_{\lambda_{1}}\}$. So $\frac{1}{|\xi|_{\mathbb{H}^{n}}^{2Q}}\chi_{\Sigma_{\lambda}^{u}}$ converges pointwise to zero as $\lambda\rightarrow\lambda_{1}$ in $\mathbb{H}^{n}\backslash (T_{\lambda_{1}}\cup \{0_{\lambda_{1}}\})$. Then, for $\lambda\in(\lambda_{1}-\delta,\lambda_{1})$, we have $\frac{1}{|\xi|_{\mathbb{H}^{n}}^{2Q}}\chi_{\Sigma_{\lambda}^{u}}\leq\frac{1}{|\xi|_{\mathbb{H}^{n}}^{2Q}}\chi_{\Sigma_{\lambda_{1}-\delta}}\in L^{1}(\Sigma_{\lambda})$. Next, the  dominated convergence theorem allows us to conclude that 
$$\int_{\Sigma_{\lambda}^{u}}\frac{1}{|\xi|_{\mathbb{H}^{n}}^{2Q}}\rightarrow0,\ as\ \lambda\rightarrow\lambda_{1}.$$
This eventually implies that 
$$C_{\lambda}\bigg(\int_{\Sigma_{\lambda}^{u}}\frac{1}{|\xi|_{\mathbb{H}^{n}}^{2Q}}d\xi\bigg)^{\frac{4}{Q}}<1.$$
Similarly, we have
$$C_{\lambda}\bigg(\int_{\Sigma_{\lambda}^{v}}\frac{1}{|\xi|_{\mathbb{H}^{n}}^{2Q}}d\xi\bigg)^{\frac{4}{Q}}<1.$$

By Lemma \ref{em1}, this implies that $U_{\lambda}\leq0$, and $V_{\lambda}\leq0$ for $\lambda\in(\lambda_{1}-\delta,\lambda_{1})$, which contradicts with the definition of $\lambda_{1}$. 

This completes the proof of Lemma \ref{le2}.
\end{proof}

\begin{lemma} \label{le3}
Let $u,v,f,g$ be as in Theorem \ref{th2} and suppose that $u,v$ are positive. Let $\bar{u},\bar{v}$ be the $CR$ inversion of $u$ and $v$ respectively. Then $\bar{u}, \bar{v}$ are symmetric with respect to $T_{0}$, that is, $\bar{u}$ and $\bar{v}$ are even functions in $t$-variable..
\end{lemma}

\begin{proof}
The proof of this lemma is verbatim to the proof of Lemma \ref{lemm5} by using Lemma \ref{le2} instead of Lemma \ref{lemm4}, so we do not repeat it. 
\end{proof}

Next we give the proof of Theorem \ref{th2}.
\begin{proof}
By Lemma \ref{le3}, we conclude that functions $\bar{u}$ and $\bar{v}$  are symmetric with respect to $T_{0},$ that is, $\bar{u}$ and $\bar{v}$ are even in $t$-variable. Consequently, $u$ and $v$ are even in $t$-variable. Since the choice of origin  is arbitrary in $t$-axis, we conclude that $u$ and $v$ are independent of $t$. Thus, $u$ and $v$ satisfy the system
\begin{equation}\label{a11}
\begin{cases}
\ -\Delta u(\xi)=f(u(\xi),v(\xi)) ,& \xi\in\mathbb{R}^{2n},\\
\ -\Delta v(\xi)=g(u(\xi),v(\xi)) ,& \xi\in\mathbb{R}^{2n}.
\end{cases}
\end{equation}
Since $f(s,t),g(s,t)$ are nondecreasing, and
$$\frac{f(s,t)}{s^{a_{1}}t^{b_{1}}}=\frac{f(s,t)}{s^{p_{1}}t^{q_{1}}}s^{p_{1}-a_{1}}t^{q_{1}-b_{1}},$$
$$\frac{g(s,t)}{s^{a_{2}}t^{b_{2}}}=\frac{g(s,t)}{s^{p_{2}}t^{q_{2}}}s^{p_{2}-a_{2}}t^{q_{2}-b_{2}},$$
where $a_{i}+b_{i}=\frac{2n+2}{2n-2}\geq\frac{Q+2}{Q-2}=p_{i}+q_{i}$, $a_{i}\geq p_{i}$ and $b_{i}\geq q_{i}, i=1,2$, then
by using \cite[Theorem 3.1]{GL08} as $2n \geq 3$ and and $h$ or $k$ is not constant function, we deduce that $(u, v)=(C_1, C_2)$ for some constant $C_1, C_2$ such that $f(C_1, C_2)=g(C_1, C_2)=0$.

This completes the proof of Theorem \ref{th2}.
\end{proof}

\section*{Acknowledgments}
RZ is supported by the National Natural Science Foundation of China (Grant No. 11871278) and the National Natural Science
Foundation of China (Grant No. 11571093).
 VK and MR are supported by the FWO Odysseus 1 grant G.0H94.18N: Analysis and Partial
Differential Equations, the Methusalem programme of the Ghent University Special Research Fund (BOF) (Grant number 01M01021) and by FWO Senior Research Grant G011522N. MR is also supported by EPSRC grant
EP/R003025/2.

\end{document}